\documentclass[draft, 12pt, reqno]{amsart}

\usepackage{tikz}
\usetikzlibrary{automata, positioning, arrows.meta}
\usepackage{subcaption}
\usepackage{graphicx} 
\usepackage[margin=1.5in]{geometry}
\usepackage{amssymb,latexsym,amsfonts,verbatim,amscd,ifthen}
\usepackage{color}
\usepackage{relsize}
\usepackage[utf8]{inputenc}
\usepackage{mathrsfs}
\usepackage{mathbbol}
\usepackage{scalefnt}
\usepackage{graphicx}

\usepackage{graphicx} 
\usepackage[margin=1.5in]{geometry}
\usepackage{amssymb,latexsym,amsfonts,verbatim,amscd,ifthen}
\usepackage{color}
\usepackage{relsize}
\usepackage[utf8]{inputenc}
\usepackage{mathrsfs}
\usepackage{mathbbol}
\usepackage{scalefnt}
\usepackage{graphicx}

\newcommand{\C}{{\mathbb C}}
\newcommand{\R}{{\mathbb R}}
\newcommand{\B}{{\mathcal B}}
\newcommand{\Alg}{\mathfrak A}

\newcommand{\Pp}{{\mathbb P}}

\newcommand{\lb}{{\lambda}}
\numberwithin{equation}{section}

\DeclareMathOperator{\aut}{Aut}

\DeclareMathOperator{\End}{End}

\DeclareMathOperator{\diag}{diag}

\DeclareMathOperator{\tr}{tr}
\DeclareMathOperator{\SP}{span}

\DeclareMathOperator{\GL}{GL}

\theoremstyle{plain}

\newtheorem{theorem}{Theorem}[section]
\newtheorem{lemma}[theorem]{Lemma}
\newtheorem{proposition}[theorem]{Proposition}
\newtheorem{corollary}[theorem]{Corollary}
\newtheorem{problem}[theorem]{Problem}
\newtheorem{example}[theorem]{Example}
\newtheorem{conjecture}[theorem]{Conjecture}

\newtheorem{remark}[theorem]{Remark} 
\newtheorem{definition}[theorem]{Definition}

\begin{document}

\title[Reducibility of linear representations]{Reducibility of linear representations, free ideals, and Kippenhahn's conjecture}
\author{Michael Stessin}
\address{Department of Mathematics and Statistics \\
University at Albany, SUNY \\
Albany, NY 12222}
\email{mstessin@albany.edu}
\author{Rongwei Yang} 
\address{Department of Mathematics and Statistics \\
 University at Albany, SUNY \\
Albany, NY 12222}
\email{ryang@albany.edu}

\thanks{The second-named author is supported in part by the
Simons Foundation Grant No. MP-TSM-00002315, the Institute for
Advanced Study Summer Collaboration Program (2025), and the American
Institute of Mathematics SQuaRE Program (2025-2027).}

\begin{abstract}
The reducibility problem is one of the central themes in representation theory. For finitely generated groups or algebras, the issue is equivalent to determining whether the tuple of linear operators representing the generating set possesses a nontrivial common invariant subspace. The last two decades have seen a rich interplay between this problem and the geometry of the projective joint spectra of such tuples. This paper presents a comprehensive study of the connection for matrix tuples and finite dimensional group representations. Among other things, several key new concepts are introduced, including minimal polynomial, spectral index, spectral stability, and characteristic graph. Notably, this framework leads to a complete resolution of an extended Kippenhahn's conjecture by providing a necessary and sufficient condition for its validity.
\end{abstract}

\maketitle

\section{Introduction}

The reducibility of linear representations of groups or algebras is one of the central topics in representation theory. For finitely generated groups or algebras, this issue amounts to determining the existence of nontrivial common invariant subspaces for the operator tuples that represent the generating sets. In the last two decades, this question was intensely studied in the framework of multivariable operator theory and projective (joint) spectrum. When the representations are finite-dimensional, the tuples consist of matrices, and the projective spectra are algebraic varieties of zero sets of 	the tuple's characteristic polynomials. An application of the classical Cayley-Hamilton theorem allows us to define some free ideals which contain a substantial amount of information about the representation. This approach also gives rise to the notion of minimal polynomial for matrix tuples, whose properties perfectly parallel those in the original one variable context. On the geometric side, this investigation gives rise to notions such as spectral index and spectral stability which are invariant under the similarity of matrix tuples.

This paper is devoted to an in-depth investigation of these topics. In the case of Hermitian tuples, which are naturally associated with unitary representations, our techniques allowed to make a substantial progress in  Kippenhahn's conjecture proposed in 1951. After a number of partial results have been established in the positive direction, a counterexample to the conjecture was found in 1983 for the general case. Theorem \ref{kippen} gives a necessary and sufficient condition for the validity of the conjecture for Hermitian tuples of arbitrary length (the original conjecture was stated for matrix pairs). This theorem substantially strengthens the partial results on the reducibility of matrix tuples obtained in \cite{S1,S2,SY}

The paper is structured as follows. Section \ref{characteristic pol} introduces the notion of joint characteristic polynomials for matrix tuples, motivated in part by the classical works of Dedekind and Frobenius. 

In Section \ref{free ideals}, we employ the Cayley–Hamilton theorem to construct free ideals associated with a matrix tuple. These ideals of group representations are examined in Section \ref{group ideals}. Section \ref{min} defines the minimal polynomial for matrix tuples, establishes its fundamental properties, and explores the relationship between a tuple's cyclicity, irreducibility, and the minimality of its characteristic polynomial. 

Section \ref{background} provides the necessary background for addressing tuple reducibility problems including Kippenhahn's conjecture. 

Section \ref{exten} is devoted to the study of tuple extensions, spectral indices and spectral multiplicities.   In particular, this section establishes a necessary and sufficient condition for the irreducibility of matrix tuples in terms of their extended characteristic polynomials (Theorem \ref{mainthm}).

Section \ref{characteristic graph} introduces the notion of characteristic graphs and reveals its intimate relation to the reducibility of a tuple. 

In Section \ref{Kippenhahn}, we state Kippenhahn's conjecture and review its history. The conjecture concerns the reducibility of Hermitian matrix pairs whose characteristic polynomials have repeated factors. We demonstrate how the conjecture, originally formulated for pairs, can be extended to tuples of arbitrary length by adding elements from the algebra generated by the tuple. Theorem \ref{kippen} here presents necessary and sufficient conditions for Kippenhahn's conjecture to be valid.  

Finally, in Section \ref{concluding}, we describe some possible directions for future developments.

\section{Characteristic polynomials}\label{characteristic pol}

Given a square matrix $A\in M_k(\C)$, its characteristic polynomial $Q_A(z)=\det (A-zI)$ is a degree-$k$ polynomial in one complex variable $z$, and it plays a fundamental role in the study of $A$. However, many problems in mathematics, science, and technology involve several matrices $A_1, ..., A_n$ of the same size. Whether there exists an appropriate notion of characteristic polynomial for such matrix tuples is a challenging problem, largely due to the vast number of possible candidates. Inspired by early work on group determinants \cite{Di75,Fr}, determinantal representations \cite{Di21,HV}, Kippenhahn's conjecture \cite{Ki,La,Law,LSS,Sh1}, and more recent research on projective spectra \cite{S,S3,S1, S2,ST,Ya09}, the following notion of characteristic polynomial for general matrices was introduced in \cite{CSZ}.

\begin{definition}
For matrices $A_1,\dots,A_n$ of the same size, their (joint) characteristic polynomial is defined as \[Q_A(z_0, z)=\det (z_0I+z_1A_1+\cdots +z_nA_n),\ \ \ z_0\in \C,\ z\in \C^n. \]
\end{definition}

This notion has been investigated intensively in recent years. For more details, we refer the reader to \cite{Ya24}. 

Related to the focus of this paper is the work by Dedekind and Frobenius in the late 1890s on group determinants. To be precise, given a finite group $G=\{1, g_1,\dots,g_n\}$, its group algebra $\C[G]$ is a $(n+1)$-dimensional complex vector space with basis $G$. Multiplication on the left by an element $g\in G$ on any vector $h\in \C[G]$ results a permutation of coefficients of $h$, and hence it gives rise to a unitary matrix in $M_{n+1}(\C)$. In current-day terminology, this is the left regular representation of the group $G$, which we shall denote by $\lambda_G$ (or simply $\lambda$). Dedekind studied the determinant 
\[Q_{\lambda}(z_0,z):=\det (z_0I+z_1\lambda(g_1)+\cdots +z_n\lambda(g_n))\]
for some groups and found it to have quite interesting factorizations. Likewise, for a general finite dimensional representations $\pi$, one may define the corresponding polynomial $Q_{\pi}(z)$. Later in 1896 Frobenius proved the following theorem \cite{De, Fr}. To express the theorem in current terminology, we let $\hat{G}$ denote the set of equivalence classes of irreducible unitary representations of $G$ (called the unitary dual of $G$). Moreover, we let $d_\pi$ denote the dimension of representation $\pi$.
\begin{theorem}\label{Frob} 
If $G=\{1, g_1,\dots,g_n\}$ is a finite group, then
\[Q_{\lambda}(z_0,z)=\prod_{[\pi]\in \hat{G}}(Q_{\pi}(z_0,z))^{d_\pi}.\]
 Moreover, each $Q_\pi$ is an irreducible polynomial, and it determines $\pi$ up to unitary equivalence.
\end{theorem}

The polynomial $Q_{\lambda}$ was later called the group determinant of $G$, and its study by Dedekind and Frobenius is indeed the starting point of group representation theory. For more information on this development, we refer the reader to \cite{Cu,Cu2, Di21, Di75, FS}. However, this line of study has not been generalized to other algebraic settings. In 2009, the notion of projective spectrum was introduced in \cite{Ya09} for general elements $A_1,\dots,A_n$ in a Banach algebra in terms of the linear pencil $A_*(z)=z_1A_1+\cdots +z_nA_n$. Some follow-up work can be found in \cite{BCY,CY,CSZ,DY,GOY,GY,HWY,MQW17,S,S4,S3,SYZ}. Recent work related to characteristic polynomial of groups can be found in \cite{CST, HY18}.

\section{Free Algebras and Ideals}\label{free ideals}

Let $\B$ denote the complex algebra of polynomials in free variables $x_1, ..., x_n$. Elements and ideals in $\B$ will be called {\em free polynomials} and respectively {\em free ideals}. Given a tuple $A=(A_1, ..., A_n)$ of bounded linear operators on a Banach space $X$, the evaluation $h(A)$ is a substitution of each $x_i$ by $A_i$ in the expression of $h$. For example, if $h(x)=2x_1x_2-x_2x_1+5x_2x_3^2x_4$, then its degree is $4$, and $h(A)=2A_1A_2-A_2A_1+5A_2A_3^2A_4$. 
\begin{definition}\label{annihi}
Given a tuple $A$ of linear operators on $X$, 

a) a free polynomial $h$ is said to {\em annihilate} $A$ if $h(A)=0$;

b) the ideal $J_A:=\{h\in \B: h(A)=0\}$ is called the {\em annihilating ideal} for $A$.
\end{definition}
It is easy to see that the evaluation map $\rho: h\to h(A)$ is an algebra homomorphism from $\B$ into $L(X)$-the algebra of all bounded linear operators on $X$. Therefore, we have $J_A=\ker \rho$ which is a two-sided ideal in $\B$.

In the case $n=1$, the algebra $\B$ is the polynomial ring $\C[x]$, and $A$ is a single square matrix. It is well-known that $J_A$ is the principal ideal generated by the minimal polynomial for $A$. In particular, the characteristic polynomial $p_A(x)=\det (x-A)$ belongs to $J_A$ due to the Cayley-Hamilton theorem. We will denote the principal ideal $(p_A)$ by $J_A^{\scalebox{0.5}{CH}}$. In infinite dimension, a single linear operator $T\in L(X)$ is said to be {\em algebraic} if there exists a polynomial $p\in \C[x]$ such that $p(T)=0$, or equivalently, the ideal $J_T$ is nontrivial.  In the sequel, we will make analogous definitions of such ideals for matrix tuples.

\subsection{Cayley-Hamilton Ideals}

 It is meaningful to start the study with a simple example, as it sheds light on subsequent discussions.

\begin{example}\label{Pmatrices} \normalfont
The $3$-dimensional complex simple Lie algebra $\frak{sl}_2(\C)$ plays an important role in mathematics and physics. It is spanned by the {\em Pauli matrices}
\[\sigma_1=\left(\begin{matrix}
0 & 1\\
1 & 0
\end{matrix}\right),\ \sigma_2=\left(\begin{matrix}
0 & -i\\
i & 0
\end{matrix}\right),\ \sigma_3=\left(\begin{matrix}
1 & 0 \\
0 & -1
\end{matrix}\right),\]
which satisfy the commutation relations:
\begin{equation}\label{Pcom}
 [\sigma_1,\sigma_2]=2i\sigma_3,\  [\sigma_2,\sigma_3]=2i\sigma_1,\  [\sigma_3,\sigma_1]=2i\sigma_2.
 \end{equation}
Setting $\sigma_* (z)=z_1\sigma_1+z_2\sigma_2+ z_3\sigma_3$, the characteristic polynomial \[Q_\sigma(z_0,z)=\det (z_0I+\sigma_*(z))=z_0^2-z_1^2-z_2^2-z_3^2.\] For each $z\in \C^3$, the pencil $\sigma_* (z)$ is a single matrix. Cayley-Hamilton theorem then implies that
\begin{align*}
0=Q(-\sigma_*(z), z)&= \sigma_*^2(z)-(z_1^2+z_2^2+z_3^2)I\\
&=\sum_{i=1}^3z_i^2(\sigma_i^2-I)+\sum_{i< j}^3z_iz_j(\sigma_i\sigma_j+\sigma_j\sigma_i),\ \ \ z\in \C^3,
\end{align*}
and consequently,
\begin{equation}\label{Pauli}
\sigma_i^2=I,\ \ \text{and}\ \ \sigma_i\sigma_j+\sigma_j\sigma_i=0, \ \ \ 1\leq i, j\leq 3,\ i\neq j.
\end{equation}
\end{example}

\begin{remark} \normalfont Some observations are noteworthy.

{\bf 1}. Equations (\ref{Pauli}) define the so-called Pauli algebra 
\[G_3:=\SP_{\mathbb R}\{1, \sigma_i, \sigma_i\sigma_j, \sigma_1\sigma_2\sigma_3: 1\leq i, j\leq 3,\ i\neq j\},\]
which is traditionally known as the Clifford algebra $Cl(3,0)$. 

{\bf 2}. What is probably more remarkable about Example \ref{Pmatrices} is this fact: if $A_1, A_2$, and $A_3$ are any $2\times 2$ matrices whose joint characteristic polynomial is $Q_A(z_0,z)=z_0^2-z_1^2-z_2^2-z_3^2$, then they must satisfy the equations (\ref{Pauli}). Hence the Pauli algebra is inherent to the quadratic form $Q_\sigma$ which defines the Minkowski metric. 

{\bf 3}. Given any homogeneous polynomial $Q(z), z\in \C^n$, the group \[\aut (Q):=\{M\in \GL(n,\C): Q(zM)=Q(z)\}\] is called the {\em invariant group} of $Q$. The Lorentz group of spacetime is the subgroup in $\aut (Q_\sigma)$ consisting of real matrices.
\end{remark}

Throughout this paper, unless specified otherwise, we assume $n\geq 2$ and $A=(A_1, ..., A_n)$ stands for a tuple of $k\times k$ matrices. Then, we can expand $Q_A(z_0, z)$ as a polynomial in $z_0$, namely, 
\begin{equation}\label{expandQ}
    Q_A(z_0, z)=z_0^k+q_1(z)z_0^{k-1}+\cdots + q_k(z),
    \end{equation}
    where each $q_j$ is a homogeneous polynomial of degree $j$, and it manifests certain properties of the matrices. For instance, it is not hard to check \cite{Ya24} that
\begin{align}
 q_1(z)&=\sum_{j=1}^n z_j\tr A_j,\\
 q_2(z)&=\frac{1}{2}\sum_{i,j=1}^n z_iz_j\left(\tr A_i \tr A_j-\tr (A_iA_j)\right),
\end{align}where $\tr$ is the trace. If we substitute $z_0$ by $-A_*(z)$, then we can write 
\begin{align}
Q_A(-A_*(z), z)&=(-A_*(z))^k+q_1(z)(-A_*(z))^{k-1}+\cdots + q_k(z)I \nonumber \\
&=:\sum_{|\tau|=k}z^{\tau}p_\tau(A),\label{Qexpan}
\end{align} where $\tau=(\tau_1, ..., \tau_n)$ is the multi-index with $0\leq \tau_i\leq k$, $|\tau|=\tau_1+\cdots +\tau_n$,  $z^\tau=z_1^{\tau_1}\cdots z_n^{\tau_n}$, and $p_\tau$ in the free algebra $\B$ for each $\tau$. Cayley-Hamilton theorem asserts that $Q_A(-A_*(z), z)=0$ for all $z\in \C^{n}$, and thus equation (\ref{Qexpan}) leads to the following multivariable version of the theorem.

\begin{theorem}\label{C-H}
For any tuple $A$, we have $p_\tau (A)=0$ for each $\tau$ with $|\tau|=k$.
\end{theorem}
The following definition is immediate.

\begin{definition}
The {\em Cayley-Hamilton ideal} of tuple $A$, denoted by $J_A^{\scalebox{0.5}{CH}}$, is the two-sided ideal in $\B$ generated by the free polynomials $p_\tau, |\tau|=k$.
\end{definition}
\noindent It is clear that $J_A^{\scalebox{0.5}{CH}}\subseteq J_A$. Writing $J_A^{\scalebox{0.5}{CH}}:=\left(p_\tau: |\tau|=k\right)$, for Example \ref{Pmatrices}, we have 
\[J_\sigma^{\scalebox{0.5}{CH}}=\left(x_i^2-1, x_ix_j+x_jx_i : 1\leq i, j\leq 3,\ i\neq j\right).\]Observe that due to (\ref{Pcom}) the free polynomial $x_1x_2-x_2x_1-2ix_3$ is in $J_\sigma$, but it is not in $J_\sigma^{\scalebox{0.5}{CH}}$. Indeed, it is not hard to verify that \[J_\sigma=(J_\sigma^{\scalebox{0.5}{CH}}, x_1x_2-x_2x_1-2ix_3).\]The following is a simple consequence of Theorem \ref{C-H}.
\begin{corollary}\label{annih}
For every tuple $A$ of $k\times k$ matrices, the ideal $J_A$ contains a nontrivial free polynomial in two or more variables with degree less than or equal to $k$.
\end{corollary}
It is worth noting that, if not for the notion of joint characteristic polynomial, Corollary \ref{annih} is not at all obvious. For any positive integer $m$, we define the free polynomial 
\[S_m(x_1, ..., x_m)=\Sigma_{\sigma\in S(m)}\text{sgn}(\sigma)x_{\sigma(1)}\cdots x_{\sigma(m)},\]where $S(m)$ stands for the group of permutations on the set $\{1, ..., m\}$.
Amitsur–Levitzki theorem \cite{AL} states that for any abelian ring $R$ we have 
\[S_{2k}(A_1, ..., A_{2k})=0,\ \ \ A_j\in M_k(R), \ j=1, ..., 2k.\] Moreover, there is no free polynomial of degree less than $2k$ that annihilate all matrix $2k$-tuples in $M_k(R)$. In Theorem \ref{C-H}, while the polynomials $p_\tau$ are specific to the tuple $A$, they are all of degrees less than or equal to $k$.

The Cayley-Hamilton ideal $J_A^{\scalebox{0.5}{CH}}$ can coincide with the ideal $J_A$ for some pairs $A=(A_1,A_2)$. 

\begin{example}\normalfont Let 
$$A_1=\left [ \begin{array}{cc} 1&0 \\0& -1 \end{array} \right ], \ A_2 =\left [ \begin{array}{cc} 0&1 \\1&0\end{array}\right ]$$
be the Coxeter generators of an irreducible representation of the dihedral group \[D_4=\langle g_1, g_2\mid g_1^2=g_2^2=(g_1g_2)^4=1\rangle.\] Computations like that in Example \ref{Pmatrices} verify that the Cayley-Hamilton ideal $J_A^{\scalebox{0.5}{CH}}=(x_1^2-1 ; \ x_2^2-1; \ x_1x_2+x_2x_1)$. It is not hard to check that every element in the quotient $J_A/J_A^{\scalebox{0.5}{CH}}$ is of the form $\alpha [x_1x_2]+\beta [x_1]+\gamma [x_2] +\delta $
for some $\alpha, \beta, \gamma, \delta \in \C$, and it satisfies the equation
$$\alpha  A_1A_2 +\beta A_1+ \gamma A_2 +\delta  I=0. $$
The second and the fourth term of the above equation are diagonal matrices, while the first and the third are off-diagonal. Thus,
$$\alpha A_1A_2 +\gamma A_2=0, \ \ \beta A_1 + \delta I=0, $$ 
which implies $\alpha=\beta=\gamma=\delta=0$.
\end{example}

\section{Free Ideals of Group Representations}\label{group ideals}

Consider a group $G$ with a finite generating set $S=\{g_1, ..., g_n\}$. For the remaining part of this paper, we assume $S$ is symmetric in the sense that $g\in S$ implies $g^{-1}\in S$. For any $g\in G$, we let $\ell(g)$ denote the word length of $g$ with respect to $S$. Any relation among the elements in $S$ can be described by a free polynomial $h\in \B$. For example, the relation $(g_1g_2)^4=1$ corresponds to $h(x)=(x_1x_2)^4-1$. For convenience, we set $\bar{g}=(g_1, ..., g_n)$, and given a representation $\pi$ of $G$, write $\pi(\bar{g}):=(\pi(g_1),..., \pi (g_n))$.

\begin{definition}
Given $G$ above, define the free ideal $J_{\scalebox{0.6}{G}}=\{h\in \B: h(\bar{g})=0\}$.
\end{definition}
\noindent Note that the dependence of $J_{\scalebox{0.6}{G}}$ on $S$ is not the concern of this paper. It is clear that the quotient algebra $\B/J_{\scalebox{0.6}{G}}$ is isomorphic to the group algebra $\C[G]$. Thus the algebraic structure of $G$ is uniquely determined by $J_{\scalebox{0.6}{G}}$.

We let $GL(X)$ denote the group of nonsingular operators on a Banach space $X$ and let $\B(X)$ denote the Banach algebra of bounded linear operators on $X$. A group representation $\pi: G\to GL(X)$ extends linearly to an algebra homomorphism from $\C[G]$ into $\B(X)$ (which we also denote by $\pi$). It is worth noting that a faithful group representation does not necessarily extend to a faithful algebraic homomorphism.
\begin{definition}
Given $G$ and $\pi$ above, we define the free ideal \[J_{\scalebox{0.6}{G},\pi}=\{h\in \B: h(\pi(\bar{g}))=0\}.\] 
In the case $\pi$ is finite dimensional, we denote the Cayley-Hamilton ideal of $\pi(\bar{g})$ by $J_{\scalebox{0.6}{G},\pi}^{\scalebox{0.6}{CH}}$.
\end{definition}
It is clear that $\pi (h(\bar{g}))=h(\pi(\bar{g}))$ for every free polynomial $h$. Thus, if $\pi: \C[G]\to \B(X)$ is faithful, then $h(\pi(\bar{g}))=0$ if and only if $h(\bar{g})=0$. The following proposition summarizes the above observation.
\begin{proposition}
Given any $G$ and $\pi$ as above, the following hold:

a) $J_{\scalebox{0.6}{G}}\subseteq J_{\scalebox{0.6}{G},\pi}$;

b) $J_{\scalebox{0.6}{G}}=J_{\scalebox{0.6}{G},\pi}$ if and only if $\pi: \C[G]\to \B(X)$ is faithful.
\end{proposition}
\noindent In particular, if $X=\ell^2(G)$ and $\lambda$ is the left regular representation of $G$, then $\lb: \C[G]\to B(X)$ is faithful, and therefore $J_{\scalebox{0.6}{G},\lb}=J_{\scalebox{0.6}{G}}$. The following facts are not hard to check.
\begin{proposition}
Let  $\pi$ and $\rho$ be unitary representations of $G$.

a) If $\pi\cong \rho$, then $J_{\scalebox{0.6}{G},\pi}=J_{\scalebox{0.6}{G},\rho}$.

b) If $\pi$ is a subrepresentation of $\rho$, then $J_{\scalebox{0.6}{G},\rho}\subseteq J_{\scalebox{0.6}{G},\pi}$.
\end{proposition}

If $\pi$ is unitarily equivalent to a subrepresentation of $\rho$, we say that $\pi$ is contained in $\rho$. The notion of weak containment is subtler. Consider two unitary representations $\pi$ and $\rho$ of $G$ in Hilbert spaces ${\mathcal H}$ and ${\mathcal K}$, respectively.

\begin{definition} We say that $\pi$ is weakly contained in $\rho$ (denoted by $\pi\prec \rho$) if for every $u\in {\mathcal H}$, every finite subset $Q\subset G$, and every $\epsilon>0$, there exist $v_1, \cdots, v_m$ in ${\mathcal K}$ such that 
\[|\langle \pi(g)u,\ u\rangle-\sum_{i=1}^{m}\langle \rho(g)v_i,\ v_i\rangle|<\epsilon,\ \ \ \forall g\in Q.\]
\end{definition}
\noindent Furthermore, $\pi$ and $\rho$ are said to be {\em weakly equivalent} if $\pi\prec \rho$ and $\rho\prec \pi$. In this case, we write $\pi\sim \rho$. For more information about weak containment, we refer the reader to \cite{BHV,Ya24}.

For any free polynomial $h\in \B$ of degree $N$, using multi-indices given by ordered words, we can write $h(x)=\sum_{|\tau|\leq N}c_\tau x_\tau$, where $|\tau |$ stands for the length of the word $\tau$. Define symbolically \[h^*(x)=\sum_{|\tau|\leq N}\overline{c_\tau} (x_\tau)^{-1},\] and observe that $h^*(\bar{g})$ is well-defined and lies in $\C[G]$. Further, if $\pi$ is a unitary representation of a group $G$, then $\pi (h^*(\bar{g}))=(\pi (h(\bar{g})))^*$.
\begin{proposition}
Let $(\pi, {\mathcal H})$ and $(\rho, \mathcal{K})$ be two unitary representations of $G$. If $\pi\prec \rho$, then $J_{G,\rho}\subseteq J_{G,\pi}$.
\end{proposition}
\begin{proof}
Let $h(x)=\sum_{|\tau|\leq N}c_\tau x_\tau$ be an element in $J_{G,\rho}$ of degree $N$. Then $h(\rho(\bar{g}))=0$. Set $Q=\{g\in G: \ell (g)\leq 2N\}$ and observe that $h^*(\bar{g})h(\bar{g})\in \SP Q\subset \C[G]$. Write 
\[h^*(\bar{g})h(\bar{g})=:\sum_{g\in Q}b_g g, \quad M:=\sum_{g\in Q}|b_g|.\]
Since $\pi\prec \rho$, for any $u\in {\mathcal H}$ and $\epsilon>0$, there exist $v_1, ..., v_m\in {\mathcal K}$ such that
\[\bigg|\langle \pi(g)u, u\rangle-\sum_{i=1}^m\langle \rho(g)v_i, v_i\rangle\bigg|<\frac{\epsilon}{(1+M)|Q|},\ \ \ \forall g\in Q.\] It follows that
\begin{align*}
\|h(\pi(\bar{g}))u\|^2&=\|h(\pi(\bar{g}))u\|^2-\sum_{i=1}^m\|h(\rho(\bar{g}))v_i\|^2\\
&=\bigg|\langle \pi(h^*(\bar{g})h(\bar{g}))u, u\rangle-\sum_{i=1}^m\langle \rho(h^*(\bar{g})h(\bar{g}))v_i, v_i\rangle\bigg|\\
&=\bigg|\sum_{g\in Q}b_g\bigg(\langle \pi(g)u, u\rangle-\sum_{i=1}^m\langle \rho(g)v_i, v_i\rangle\bigg)\bigg|\\
&\leq M\sum_{g\in Q}\bigg|\langle \pi(g)u, u\rangle-\sum_{i=1}^m\langle \rho(g)v_i, v_i\rangle\bigg|< \epsilon,
\end{align*}
which implies $h(\pi(\bar{g}))u=0$. Since $u$ is arbitrary, we have $h(\pi(\bar{g}))=0, i.e., h\in J_{G,\pi}$.\end{proof}
\noindent As an immediate consequence, if $\pi\sim \rho$, then $J_{G,\rho}= J_{G,\pi}$.
 
\section{Minimal Polynomials in Several Variables}\label{min}

Given a tuple of $k\times k$ matrices $A=(A_1, ..., A_n)$, their characteristic polynomial $Q_A(z_0,z)$ is homogeneous of degree $k$, and, as we saw above,  the Cayley-Hamilton theorem implies $Q_A(-A_*(z),z)=0$ for all $z\in \C^n$. It is a natural question whether there exists a polynomial $q(z_0, z)\in \C[z_0, ..., z_n]$ of degree less than $k$ such that $q(-A_*(z), z)=0$ in $\C^n$. This motivates the following definition.
\begin{definition}\label{defminimal}
 Let $q\in \C[z_0, ..., z_n]$ be a homogeneous polynomial that is monic in $z_0$. We say that $q$ is {\em minimal} for the tuple $A$ if 
    
    a) $q(-A_*(z),z)=0$ for all $z\in \C^n$, and

    b) there exists no such polynomial of lower degree.
\end{definition}
\noindent A distinction is worth noting. For a matrix tuple, its annihilating polynomials are free polynomials (Definition \ref{annihi}), while its minimal polynomial (as defined above) is a homogeneous polynomial in several complex variables. For example, the free polynomials $x_i^2-1, x_ix_j+x_jx_i, 1\leq i, j\leq 3, i\neq j$ are annihilating polynomials for the tuple $\sigma$ of Pauli matrices, while the polynomial $z_0^2-z_1^2-z_2^2-z_3^2$ is the minimal polynomial for $\sigma$. Furthermore, in light of Theorem \ref{C-H}, the minimal polynomial of a tuple $A$ gives rise to several annihilating polynomials of $A$ with degrees less than or equal to the degree of the minimal polynomial. 

\subsection{Cyclicity}

Recall that a subalgebra $\mathfrak{A}$ of the algebra of bounded linear operators acting on a Banach space $X$ is said to be cyclic, if there is a vector $x\in X$ such that $\mathfrak{A}x$ is dense in $X$. Of course, if $X=\C^k$, then $\mathfrak{A}$ is a subalgebra in $M_k(\C)$. Given a $k\times k$ matrix $T$, the Cayley-Hamilton Theorem implies that the algebra generated by $T$ is $\mathfrak{A}(T)=\SP\{I, T, ..., T^{k-1}\}$. We say that $T$ is cyclic, if $\mathfrak{A}(T)$ is.
Thus, matrix $T$ is cyclic if and only if $\{I, T, ..., T^{k-1}\}$ are linearly independent, and this occurs if and only if its characteristic polynomial is minimal, for instance see \cite[Sect. 7.1]{HK}. 

Since this condition can be expressed as a polynomial inequality in the entries of $T$, the set of cyclic matrices in $M_k(\C)$ is open. We state this well-known fact as a lemma. Details can be found in the proof of Theorem \ref{lcyclic}. 
\begin{lemma}\label{cyclicM}
The set of cyclic matrices is open in $M_k(\C)$.
\end{lemma}
For a matrix tuple $A$, let ${\mathfrak A}(A)$ denote the unital subalgebra of $M_k(\C)$ generated by $I, A_1, ..., A_n$. We say that $A$ is {\em cyclic} if there exists a vector $v\in \C^k$ such that ${\mathfrak A}(A)v=\C^k$. Such $v$ will be called a {\em cyclic vector} for the tuple $A$. Furthermore, we say that $A$ is {\em linearly cyclic} if there exists a vector $w\in \C^n$ such that $A_*(w)=w_1A_1+\cdots +w_nA_n$ is cyclic. Apparently, in the case $n=1$, cyclicity coincides with linear cyclicity. When $n\geq 2$, linear cyclicity is a stronger condition on the tuple $A$ than cyclicity. 

We say that the tuple $A$ is {\em reducible} if there exists a nonzero vector $v$ that is not cyclic for tuple $A$. In this case $M:={\mathfrak A}(A)v$ is a closed proper subspace of $\C^k$ that is invariant for all $A_j$. In this case, the space $\C^k$ can be decomposed into an orthogonal direct sum $M\oplus M^\perp$, with respect to which each matrix $A_j$ has the upper triangular block form 
\begin{equation}\label{decom}
\begin{pmatrix} A_j' & \ast \\ 0 & A_j'' \end{pmatrix},\ \ \ j=1, ...,n.
\end{equation}
We say that $A$ is {\em unitarily reducible} if the matrices $A_j$ have a nontrivial common reducing subspace. Otherwise, we say that $A$ is {\em irreducible} and resp. {\em unitarily irreducible}.

Suppose $A$ is reducible. Letting $A'$ and $A''$ denote the tuples $(A'_1, ..., A'_n)$ and $(A''_1, ..., A''_n)$, respectively, we obtain a factorization of the characteristic polynomials: $Q_A=Q_{A'}Q_{A''}$. Therefore, if $Q_A$ is irreducible, then tuple $A$ is irreducible, and consequently $\mathfrak{A}(A)=M_k(\C)$ due to Burnside's theorem \cite{Bur,LR}. Thus, the characteristic polynomial $Q_A$ offers a simple sufficient condition for determining the irreducibility of $A$. However, this condition is far from being necessary. Kippenhahn's conjecture \cite{Ki} attempted to address this issue for Hermitian matrix tuples. We will take a look at this conjecture in Section \ref{Kippenhahn}.

\begin{example}
    Let $G$ be a group as before and consider an irreducible representation $\pi: G\to U(k)$, where $U(k)$ stands for the group of unitary $k\times k$ matrices. Since $\pi(g)$ is unitary and hence normal, the tuple $\pi(\overline{g})$ is irreducible if and only if it is unitarily irreducible.
\end{example}

Definition \ref{defminimal} has the following immediate consequence.
\begin{lemma}\label{minip}
Suppose $q(z_0,z)$ is a degree-$d$ minimal polynomial of tuple $A$. Then for each fixed $w\in \C^n, w\neq 0$, the minimal polynomial of $A_*(w)$ has degree no more than $d$.
\end{lemma}
   
\noindent In particular, if $A$ is linearly cyclic, then there exists a complex vector $w\in \C^n$ such that $A_*(w)$ is cyclic. Hence the characteristic polynomial of the matrix $A_*(w)$ is minimal and has degree $k$. Lemmas \ref{cyclicM}, \ref{minip}, and Riemann extension theorem thus lead to the following result. 

\begin{theorem}\label{lcyclic}
The characteristic polynomial $Q_A$ is minimal if and only if $A$ is linearly cyclic.
\end{theorem}
\begin{proof}
If $A$ is linearly cyclic, then there exists a vector $w\in \C^n$ for which $A_*(w)$ is cyclic, implying that the minimal polynomial for $A_*(w)$ coincides with its characteristic polynomial which has degree $k$. It follows from Lemma \ref{minip} that the minimal polynomial of $A$ has degree $k$, showing that $Q_A$ is minimal.

To prove the necessity, we show that if tuple $A$ is not linearly cyclic then $Q_A$ is not minimal.
First, we let $d(z)$ denote the degree of the minimal polynomial for $A_*(z)$ and set \[r_A:=\max \{d(z): z\in \C^n\}.\] For convenience, we shall write $r_A$ often as $r$. If $d(w)=k$ for some $w\in \C^n$, then $A_*(w)$ is cyclic, contradicting the assumption that $A$ is not linearly cyclic. Thus, we assume $r=d(w)<k$. This implies that the set $\{I, A_*(w), ..., A^{r-1}_*(w)\}$ is linearly independent. For any matrix $T\in M_k(\C)$, we let vec$(T)$ denote the vector in $\C^{k^2}$ in which the $j$th block is the $j$th column of $T$. Then the $k^2\times r$ matrix 
\[M(w):=\begin{pmatrix} \text{vec}(I) & \text{vec}(A_*(w)) & \cdots & \text{vec}(A^{r-1}_*(w))\end{pmatrix}\] has rank $r$, and consequently the matrix $M^*(w)M(w)$ is a $r\times r$ invertible matrix, where $M^*$ stands for the adjoint of matrix $M$. It follows that there is an open neighborhood $U$ of $w$ on which $M^*(z)M(z)$ is invertible for all $z\in U$. Therefore, the set $\Omega:=\{z\in \C^n: d(z)=r\}$ is open. Note that its complement \[\Omega^c=\{z\in \C^n: \det \big(M^*(z)M(z)\big)=0\}\] consists of points $z$ such that all $r\times r$ minors of $M(z)$ have determinant $0$ and thus is an algebraic variety.

Suppose 
\[p(z_0,w)=z_0^r-a_{r-1}(w)z_0^{r-1}-\cdots -a_0(w)=\prod_{j=1}^r(z_0-\lambda_j(w)),\ \ \ w\in \Omega,\] is the minimal polynomial for $A_*(w)$. Then the fact $p(A_*(w),w)=0$ implies
\[a_{r-1}(w)\text{vec}(A^{r-1}_*(w))+\cdots +a_0(w)\text{vec}(I)=\text{vec}(A^r_*(w)),\ \ \ w\in \Omega.\]Picking a $r\times r$ rank-$r$ minor of $M(w)$ and solving the coefficients $a_i$ by the Cramer's rule, we see that each $a_j$ is a rational function in $w$. Note that due to uniqueness, the solutions $a_i$ are independent of the choice of the $r\times r$ rank-$r$ minor. Since each $\lambda_j(w)$ is an eigenvalue of $A_*(w)$, and $A_*(tw)=tA_*(w),\ t\in \C$, we have $\lambda_j(tw)=t\lambda_j(w)$. Thus the rational functions $a_i(w)$ are homogeneous of degree $r-i$ with possible singularities lying inside $\Omega^c$. Furthermore, observe that the norm \[\|A_*(w)\|\leq \|w\|\sqrt{\|A_1\|^2+\cdots +\|A_n\|^2},\] which is bounded in every bounded subset of $\C^n$. Since each $a_j(w)$ is a symmetric polynomial in the eigenvalues $\lambda_1(w), ..., \lambda_r(w)$ of $A_*(w)$, it is also bounded in every bounded subset of $\C^n$ and hence must be a polynomial. Setting $q(z_0,z)=(-1)^rp(-z_0,z)$, we have $q(-A_*(w),w)=0$ for all $w\in \Omega$ and hence is constant $0$ in $\C^n$, showing that the characteristic polynomial $Q_A$ is not a minimal.
\end{proof}

\subsection{Minimality}
The following corollaries are perfect several variable extensions of classical results for single matrices.

\begin{corollary}\label{properties}
The following hold for all matrix tuples $A$:

a) The minimal polynomial is unique.

b) The minimal polynomial is a factor of $Q_A$.
\end{corollary}
\begin{proof}

For part a), suppose $p$ and $p'$ are minimal polynomials of $A$ with degree $r$. Then the proof of Theorem \ref{lcyclic} shows that there exists a vector $w\in \C^n$ such that the set $\{I, A_*(w), ..., A^{r-1}_*(w)\}$ is linearly independent. Moreover, the set of all such vectors $w$ is equal to $\C^n\setminus V$ for some algebraic variety $V\subset \C^n$. We write
\begin{align*}
  p(z_0,z)&=z_0^r+z_0^{r-1}q_1(z)+\cdots +q_{r}(z),\\
  p'(z_0,z)&=z_0^r+z_0^{r-1}q'_1(z)+\cdots +q'_{r}(z),
  \end{align*} 
where $q_j$ and $q_j'$ are homogeneous polynomials of degree $1\leq j\leq r$ in $z\in \C^n$. Then 
\begin{equation}\label{Q-p}
0=(p-p')(-A_*(z),z)=(-A_*(z))^{r-1}(q_1-q'_1)(z)+\cdots +(q_r-q'_r)(z)I,\ \ \ z\in \C^n.
\end{equation}
If $p\neq p'$, then there exists a $w\in \C^n\setminus V$ such that $q_j(w)\neq q'_j(w)$ for some $j$. It thus follows from \eqref{Q-p} that the set $\{I, A_*(w), ..., A^{r-1}_*(w)\}$ is linearly dependent, which is a contradiction.

For part b), the claim is trivial if the minimal polynomial $p=Q_A$. Hence we assume that $r=\text{deg}(p)<k$. Divide \(Q_A\) by the monic \(p\), we obtain \[Q_A=gp+R,\qquad \deg_{z_0}R<r,\] with $g, R\in \C[z_0,z]$. Evaluating at \(-A_*(z)\) shows that \(R\) annihilates the pencil $-A_*(z)$. On the maximal-rank locus $\Omega$, the matrices \(I,A_*(z),\ldots,A_*(z)^{r-1}\) are independent, so every scalar coefficient of \(R(-A_*(z),z)\) is zero there. Polynomial uniqueness then gives \(R=0\) identically in $\C^{n+1}$.
\end{proof}

A direct consequence of Corollary \ref{properties} is worth noting.
\begin{corollary}\label{admi1}
Given any tuple $A$, if $Q_A$ is irreducible, then it is minimal.
\end{corollary}

In light of Corollary \ref{properties} a), we denote the minimal polynomial for a matrix tuple $A$ by $q_{\scalebox{0.6}{A}}$. Define the set 
\[I_{A}:=\{p(z_0,z)\in \C[z_0,z]: p(-A_*(z),z)=0, \ \forall z\in \C^n\}.\] It is not hard to see that $I_{A}$ is an ideal of $\C[z_0,z]$. The following result is a simple consequence of polynomial division. 
\begin{corollary}\label{q_A}
For every matrix tuple $A$, we have $I_{A}=q_{\scalebox{0.6}{A}}\C[z_0,z]$.
\end{corollary}

\begin{proof}
We assume $\deg q_{\scalebox{0.6}{A}}=r$. For any $f\in I_A$, the division with remainder by the monic polynomial $q_{\scalebox{0.6}{A}}$ gives $f=gq_{\scalebox{0.6}{A}}+R$, where $g, R\in \C[z_0,z]$ with $\deg_{z_0}R=d<r$. Writing $R(z_0,z)=z_0^dr_0(z)+\cdots +r_d(z)$, where $r_j\in \C[z]$, we have 
\begin{align*}
R(-A_*(z),z)&=(-1)^dA^d_*(z)r_0(z)+\cdots +r_d(z)I\\
&=(f-gq_{\scalebox{0.6}{A}})(-A_*(z),z)=0,\ \ \ z\in \C^n.\end{align*}
Since $d<r$, and the matrices $I, A_*(z), ..., A^{r-1}_*(z)$ are linearly independent on the maximal-rank locus $\Omega$, we must have $r_j=0$ in $\Omega$ and hence in $\C^n$ for each $j=0, 1, ..., d$.

\end{proof}

In the sequel, a matrix tuple $A=(A_1, ..., A_n)$ is said to be {\em Hermitian} if each matrix $A_j$ is Hermitian. It is well-known that the characteristic polynomial of a single Hermitian matrix $T$ is minimal if and only if each eigenvalue of $T$ has multiplicity 1. The next theorem is a multivariable extension of this fact. Given any polynomial $Q(z_0,z)$ with irreducible factorization
\begin{equation}\label{char}
Q(z_0,z)= \displaystyle \prod_{j=1}^s R_j(z_0,z)^{t_j}, \ \ \ t_j\geq 1, 
\end{equation}
where each $R_j$ is monic in $z_0$, we define $\operatorname{rad}(Q)= \prod_{j=1}^s R_j$.

\begin{theorem}\label{chacteristic minimal}
Let $A$ be a Hermitian tuple. Then $q_{\scalebox{0.6}{A}}=\operatorname{rad}(Q_A)$.
\end{theorem}

\begin{proof}
For every real parameter \(x\),  pencil \(-A_*(x)\) is diagonalizable and every one of its eigenvalues is a zero of the square-free part of \(Q_A(z_0,x)\). Thus $\operatorname{rad}(Q_A)$ annihilates \(-A_*(x)\). Polynomial uniqueness extends the identity from \(\mathbb R^n\) to \(\mathbb C^n\). Conversely, an annihilating polynomial must vanish at every characteristic root, so it must contain every distinct irreducible factor of \(Q_A\). This proves the theorem.
\end{proof}

 \noindent In particular, for Hermitian tuples $A$, the characteristic polynomial $Q_A$ is minimal if and only if it has no repeated factors.

\subsection{A Note on Frobenius' Theorem}
Let $G$ be a finite group with generating set $\{g_1, ..., g_n\}$ and $\pi$ be a unitary representation of $G$. Recall that the characteristic polynomial of $G$ associated with $\pi$ is the characteristic polynomial of the tuple $\pi (\bar{g})$, for which we denote simply by $Q_\pi(z_0, z)$. Set $A_{\pi, *}(z)=z_1\pi(g_1)+\cdots +z_n\pi(g_n)$.
Since $G$ is finite, its dual $\hat{G}$ is a finite set. When $G=\{1, g_1, ..., g_n\}$, in light of Theorem \ref{Frob}, we set $q_{\scalebox{0.6}{G}}(z_0,z):=\operatorname{rad}(Q_\lb)(z_0,z)=\prod_{[\pi]\in \hat{G}}Q_\pi(z_0, z)$. Recall that $\lb_{\scalebox{0.6}{G}}$ ($\lb$ for short) stands for the left regular representation of $G$. The following result sheds new light on Frobenius' theorem.
\begin{proposition}\label{minimal}
 Assume $G=\{1, g_1, ..., g_n\}$.  Then $q_{\scalebox{0.6}{G}}$ is the minimal polynomial for the tuple $\lb(\bar{g})$. 
 \end{proposition}
\begin{proof}
 First, we check that $q_{\scalebox{0.6}{G}}(-A_{\lb, *}(z), z)=0, z\in \C^n$. It is known that the regular representation is unitarily equivalent to a direct sum of irreducible representations, i.e., $\lb\cong\oplus_{[\pi]\in \hat{G}}\pi^{d_\pi}$, where $\pi^{d}$ stands for the direct sum of $d$ copies of $\pi$. Then 
 \begin{equation}\label{Alb}
 A_{\lb,*}(z)=\bigoplus_{[\pi]\in \hat{G}}\bigg(\bigoplus_{j=1}^{d_\pi}A_{\pi,*}(z)\bigg).
 \end{equation}
 \noindent It follows that
\begin{align*}
q_{\scalebox{0.6}{G}}(-A_{\lb,*}(z),z)&=\bigoplus_{[\pi]\in \hat{G}}\bigg(\bigoplus_{j=1}^{d_\pi}q_{\scalebox{0.6}{G}}(-A_{\pi,*}(z),z)\bigg)\\
    &=\bigoplus_{[\pi]\in \hat{G}}\bigg(\bigoplus_{j=1}^{d_\pi}\prod_{[\rho]\in \hat{G}} Q_\rho(-A_{\pi,*}(z),z)\bigg).
    \end{align*}
Since $Q_\pi(-A_{\pi,*}(z),z)=0$ by Cayley-Hamilton theorem, all summands above are equal to $0$, giving $q_{\scalebox{0.6}{G}}(-A_{\lb,*}(z),z)=0, z\in \C^n$.

Next, we check the minimality of $q_{\scalebox{0.6}{G}}$. Theorem \ref{Frob} indicates that all factors $Q_\pi$ of $q_{\scalebox{0.6}{G}}, [\pi]\in \hat{G}$, are irreducible. Moreover, we have $Q_\pi\neq Q_\rho$ if and only if $[\pi]\neq [\rho]$. If $q$ is the minimal polynomial of the tuple $\lb(\bar{g})$, then Corollary \ref{q_A} implies that $q\mid q_{\scalebox{0.6}{G}}$. Suppose that $\deg q<\deg q_{\scalebox{0.6}{G}}$. Then at least one of the factors, say $Q_{\pi_0}$ is not a factor of $q$. Since $q(-A_{\lb,*}(z),z)=0$ for all $z\in \C^n$, decomposition \eqref{Alb} implies that $q(-A_{\pi,*}(z),z)=0$ for every $[\pi]\in \hat{G}$. In particular, we have $q(-A_{\pi_0,*}(z),z)=0$. Since $Q_{\pi_0}$ is irreducible, it is minimal by Corollary \ref{admi1}, and it follows from Corollary \ref{q_A} that $Q_{\pi_0}$ must divide $q$, which is a contradiction. This completes the proof.
\end{proof}

Proposition \ref{minimal} establishes a one-to-one correspondence between the unitary dual $\hat{G}$ and the irreducible factors of the minimal polynomial $q_{\scalebox{0.6}{G}}$. However, in the case $\{1, g_1, ..., g_n\}$ is a proper subset of $G$, the characteristic polynomial $Q_{\pi}$ is not necessarily irreducible \cite{KV}, and hence the proof above may not work in general. Nevertheless, for a large class of groups, for instance all finite abelian groups and Coxeter groups \cite{CST,Ya24}, polynomial $Q_\pi$ remains irreducible whenever $\pi$ is irreducible.   

\subsection{A Summary} We close this section with a diagram that summarizes the relations among reducibility, cyclicity, and minimality. Consider the following statements:

\begin{itemize}
\item[(a)] \textit{Tuple $A$ is irreducible.}	
\item[(b)] \textit{$Q_A$ is irreducible.}
\item[(c)] \textit{$Q_A$ is minimal. }
\item[(d)] \textit{The tuple $A$ is linearly cyclic.}
\item[(e)] \textit{The algebra $\mathfrak{A}(A)$  is cyclic.}
\end{itemize}

In Fig. 1 below, the fact that Statement (Y) follows from Statement (X) is denoted by (X)$\to $ (Y).
\begin{theorem}
The diagram in Figure 1 holds for every matrix tuple $A$.	
\end{theorem}

\begin{proof}
First, the equivalence of (c) and (d) is established in Theorem 5.5. The fact that (b)$\to$ (a) is noted in Sect. 5.1, and that (a)$\to$ (e) is due to Burnside's theorem, which asserts that $\mathfrak{A}(A)=M_k(\C)$ in this case. The fact that (b) $\to$ (c) is Corollary 5.7, and ``(d) $\to$ (e)" is obvious.

To check that no other directed edges could be added to the diagram in Fig. 1, we first verify that (a) does not imply (c), and (c) does not imply (a). The former is shown later in Example \ref{Waterhouse}, which is a counterexample to Kippenhahn's conjecture. To check the latter, consider the following matrices and vectors:
\[A_1=\begin{pmatrix} 1 & 0 \\ 0 & 0 \end{pmatrix},\ \ \ A_2=\begin{pmatrix} 0 & 1\\ 0 & 0 \end{pmatrix}, \ \ \ v_1=\begin{pmatrix} 0 \\ 1 \end{pmatrix},\ \ \ v_2=\begin{pmatrix} 1 \\ 0 \end{pmatrix}.\] It is not hard to verify that $\mathfrak{A}(A)=\text{span}\{I, A_1, A_2\}$, and it has cyclic vector $v_1$ and common eigenvector $v_2$. In this case $Q_A=z_0(z_0+z_1)$, which is minimal for $(A_1, A_2)$, but $(A_1, A_2)$ is reducible because the pair has a common eigenvector. This shows that (c) does not imply (a). Furthermore, observe that this example also shows that (e) does not imply (a) or (b), and (c) does not imply (a) or (b). Finally, the fact that (a) does not imply (b) and (e) does not imply (d) will be established in Example \ref{Waterhouse}.

\end{proof}

\begin{figure}[htbp]
    \centering
    \begin{tikzpicture}[>=Stealth, node distance=2.5cm, main/.style = {draw, circle, font=\large, minimum size=8mm}]
        \node[main] (b) at (0, 4) {b};
        \node[main] (c) at (5, 4) {c};
        \node[main] (a) at (2.5, 2) {a};
        \node[main] (e) at (0, 0) {e};
        \node[main] (d) at (5, 0) {d};
        
        \path[->, thick]
            (b) edge (e)
            (b) edge (c)
            (b) edge (a)
            (a) edge (e)
            (d) edge (e)
            (d) edge [bend left=12] node {} (c)
            (c) edge [bend left=12] node {} (d);
    \end{tikzpicture}
    \caption{A summary of relations}
\end{figure}
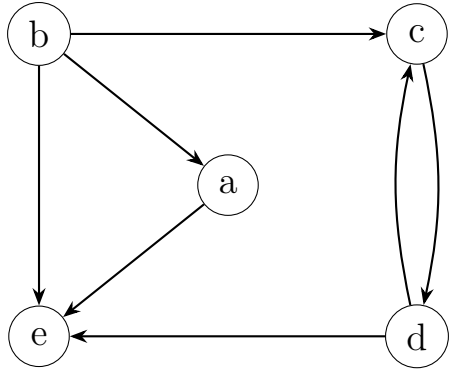

There is also no link between Statements (a) and (c) for group representations.
\begin{example}\normalfont
Consider
the following two irreducible representations of the Dihedral group $D_5$:
\begin{eqnarray*}
&\rho_1(g_1)=\left [ \begin{array}{cc} 1&0\\0&-1\end{array} \right ]	, \ \rho_1(g_2)=\left [\begin{array}{cc} \cos \ \frac{2\pi }{5}& \sin \ \frac{2\pi }{5}\\ \sin \ \frac{2\pi }{5} & -\cos \ \frac{2\pi }{5} \end{array} \right ] \\
& \rho_2(g_1)=\left [ \begin{array}{cc} 1&0\\0&-1\end{array} \right ],  \ \rho_2(g_2)=\left [\begin{array}{cc} \cos \ \frac{4\pi }{5}& \sin \ \frac{4\pi }{5}\\ \sin \ \frac{4\pi }{5} & -\cos \ \frac{4\pi }{5} \end{array} \right ]. 
\end{eqnarray*}
Write $G_1=(\rho_1(g_1),\rho_1(g_2)), \  G_2=(\rho_2(g_1),\rho_2(g_2)) $.  A simple computation (see \cite{CST} for example) shows that 
\begin{eqnarray*}
&Q_{G_1}(z_0,z_1,z_2)=z_0^2-z_1^2-2\cos \frac{2\pi }{5}z_1z_2-z_2^2, \\
&Q_{G_2}(z_0,z_1,z_2)=z_0^2-z_1^2-2\cos \frac{4\pi }{5}z_1z_2-z_2^2.	
\end{eqnarray*}
Let $A=(A_1,A_2)$ be the direct sum of $G_1$ and $G_2$,
\small\begin{equation}\label{from dihedral}
A_1=\left [ \begin{array}{cccc} 1& 0&0&0 \\ 0&-1&0&0\\0&0&1&0 \\0&0&0&-1\end{array}\right ], \ A_2=\left [ \begin{array}{cccc} \cos \frac{2\pi }{5}& \sin \frac{2\pi }{5}&0&0 \\ \sin \frac{2\pi}{5}&-\cos \frac{2\pi }{5}&0&0\\0&0&\cos \frac{4\pi }{5}&\sin \frac{4\pi }{5} \\0&0&\sin \frac{4\pi }{5}&-\cos \frac{4\pi }{5}\end{array}\right ].
 \end{equation}\normalsize
 Then the characteristic polynomial
 $$Q_A(z_0,z_1,z_2)=\big(z_0^2-z_1^2-2\cos \frac{2\pi }{5}z_1z_2-z_2^2\big)\big(z_0^2-z_1^2-2\cos \frac{4\pi }{5}z_1z_2-z_2^2\big) $$
 has 2 factors, each of multiplicity 1, so it is minimal by Theorem \ref{chacteristic minimal} despite the fact that $(A_1,A_2)$ is reducible. 
 \end{example}

\section{Some background results}\label{background}

For the remaining part of the paper we will be mostly concentrating on the reducibility of matrix tuples. In this section, we present some results obtained in previous work that we will need.

\subsection{Projections}

Given a single diagonalizable $k\times k$ matrix $A$, we denote by $\{\lambda_1,...,\lambda_r\}$ the set of its distinct eigenvalues having multiplicities $l_1,...,l_r$. respectively. Define
\begin{equation}\label{projections}
\mathcal{P}_j= \frac{1}{\displaystyle \prod_{i\neq j}(\lambda_j-\lambda_i)}\prod_{i\neq j}(A-\lambda_i I_k), \ j=1,...,r.	
\end{equation} Then $\mathcal{P}_j\in \mathfrak{A}(A)$ for each $j$. A matrix $\mathcal{P}$ is said to be a {\em projection} if $\mathcal{P}^2=\mathcal{P}$, and it is said to be an orthogonal projection if it is also self-adjoint. The following proposition is a simple consequence of functional calculus.

\begin{proposition}[\cite{SY}]\label{spectral projections}
Assume $A$ is invertible. Then $\mathcal{P}_j$ is the projection of rank $l_j$ on the $\lambda_j$-eigenspace of $A$, \ $j=1,...,r$, where
\begin{equation}\label{integral projection} 
{\mathcal P}_j=\frac{1}{2\pi i}\int_{\gamma_j} (w-A)^{-1}dw=\frac{1}{2\pi i} \int_{\gamma_j^\prime}\left(w-\frac{1}{\lambda_j}A\right)^{-1}dw,
\end{equation}
and $\gamma_j$ is a contour that separates $\lambda_j$ from all other eigenvalues, and $\gamma_j^\prime=\frac{1}{\lambda_j} \gamma_j$. 
\end{proposition}

\begin{remark}
\normalfont Of course, if $A$ is Hermitian, the projections $\mathcal{P}_j$ are orthogonal. 
\end{remark}

\subsection{Admissible Transformations}

The notion of admissible transformations of tuples was introduced in \cite{S1, SY}. Consider a tuple $A$ whose characteristic polynomial $Q_A(z_0,z)$ has an irreducible decomposition \eqref{char}. For any fixed invertible matrix $C=[c_{ij}]_{i,j=1}^n$, i.e., $C$ in the general linear group $GL(n,\C)$, consider the following tuple transformation:
\begin{equation}\label{transform}
\widehat{A}_j=\displaystyle \sum_{s=1}^n c_{js}A_s, \ j=1,...,n. 
\end{equation} 
\noindent It is easy to see that the characteristic polynomial of the tuple $\widehat{A}=(\widehat{A}_1,...,\widehat{A}_n)$ is given by
\begin{equation}\label{irrfactor}
Q_{\widehat{A}}(z_0,z)= \displaystyle \prod_{j=1}^s \widehat{R}_j(z_0,z)^{t_j}=\displaystyle \prod_{j=1}^s R_j(z_0,zC)^{t_j}=Q_A(z_0, zC). 
\end{equation}
Evidently, the factors $\widehat{R}_j$ are irreducible if and only if $R_j$ are. Observe that if we shift $A_j$ to $\widetilde{A}_j=A_j+a_jI_k$, where $a_j\in \C$, then 
\begin{equation}\label{char1}
Q_{\widetilde{A}}(z_0,z)=Q_A(z_0+\langle a,z\rangle, z)=\displaystyle \prod_{j=1}^s \widetilde{R}_j(z_0,z)^{t_j}= \prod_{j=1}^s R_j(z_0+\langle a,z\rangle,z)^{t_j},	
\end{equation}
where $\langle z,a\rangle=a_1z_1+\cdots +a_nz_n$. Similarly, the factor $\widetilde{R}_j$ is irreducible if and only if $R_j$ is. Thus, $Q_{\widetilde{A}}$ is factored with the same multiplicity of factors as $Q_{A}$ does. Since the lattice of common invariant subspaces of $A_j$ does not change after adding a scalar multiple of the identity, without loss of generality we may assume that all  matrices in the tuple $A$ are invertible, and that is what we do in the sequel.

First, assume $Q_A(z_0,z)$ has an irreducible decomposition \eqref{char} and set
\[p_A(z)=\prod_{j=1}^s R_j(-1,z).\] We let 
\begin{equation}\label{L}
L_j:=\big\{ (0,\cdots,z_j, 0,\cdots,0): \ z_j\in \C \big\}\subset \C^{n}
\end{equation} be the $j$th complex coordinate line of $\C^n$.
Recall that for an irreducible polynomial $R(z)\in \C[z_1, ..., z_n]$, a point $w$ in the algebraic variety $\{z\in \C^n: R(z)=0\}$ is said to be {\em regular} if the partial derivatives $\frac{\partial R}{\partial z_j}(w)\neq 0$ for some $1\leq j\leq n$. Equivalently, point $w$ is a regular point in the variety if the differential $dR$ does not vanish at $w$.

\begin{definition}\label{admissible}
Given a matrix $C\in GL(n,\C)$, we say that the transformation \eqref{transform} {\em admissible} if
\begin{itemize}
\item[a)] $C$ is real-valued in the case $A$ consists of Hermitian matrices;
\item[b)] for each $1\leq i\leq n$, at every point $\lambda\in \sigma_p(\widehat{A})\cap L_i$, the partial derivative $\frac{\partial {p}_{\widehat{A}}(z)}{\partial z_i}|_{z=\lambda}\neq 0$.
\end{itemize}	
\end{definition}

\noindent It is worth noting that part b) above implies that $\lambda$ is a regular point of the algebraic variety \[\sigma_p(\widehat{A})=\{z\in \C^n: p_{\widehat{A}}(z)=0\}.\]
The matrix $C$ above will be called an admissible matrix for the tuple $A$. In particular, if the identity matrix $I_n$ is admissible for $A$, then we say that tuple $A$ is admissible. In this case, since every regular point of the variety $\sigma_p(A)$ has multiplicity 1, we see that the intersection of $\Gamma_p=\big\{z\in \C^n: R_p(z)=0\}$ with $L_j$ consists of $deg(R_p)$ distinct points, and these sets of points are different for different $p$. 

\begin{remark}\label{sim}
\normalfont Another observation is that whether a tuple is admissible or not depends only on the tuple's characteristic polynomial, but not on the tuple itself. For example, let $A$ be an admissible tuple of $k\times k$ matrices, and $D\in GL(k,\C)$. In sections \ref{characteristic graph} and \ref{Kippenhahn} below we will be passing from $A$ to the similar tuple $A^D=(DA_1D^{-1},...,DA_nD^{-1})$. Since $Q_A=Q_{A^D}$, the tuple $A$ is admissible if and only if $A^D$ is.
\end{remark}

The following theorem was proved for Hermitian tuples in \cite{S1} and for non-Hermitian tuples in \cite{SY}.

\begin{theorem}[\cite{S1, SY}]\label{adm}
For every tuple of invertible matrices $(A_1,...,A_n)$, the set of admissible transformations $C$ is open and dense in every neighborhood of the identity matrix $I\in M_n(\C)$.	
\end{theorem}
\noindent Note that if $(A_1,...,A_n)$ is Hermitian, then the set of real transformation matrices is open and dense in every neighborhood of $I\in M_n({\mathbb R})$.

The next fact is a consequence of \eqref{char1}, the remarks following Definition \ref{admissible}, and Theorem \ref{adm}.
\begin{corollary}\label{admi}
Given any matrix $n$-tuple $A$ with $Q_A$ being factored as in \eqref{char}, there exists a nonempty Zariski-open set $\mathcal{O}\subset \C^n$ such that for every $w\in \mathcal{O}$, the one variable polynomials $R_j(z_0,w), 1\leq j\leq s$, have simple and pairwise disjoint zeros.
\end{corollary}

\subsection{Local Spectral Analysis for Matrix Tuples}
Now we consider a tuple $A=(A_1,...,A_n)$ of $k\times  k$ matrices and define the \textit{proper projective joint spectrum} as in \cite{ST}:
\begin{eqnarray}
\sigma_p(A_1,...,A_n)=\{z\in \C^n: \ \det(I-A_*(z))=0\}. \label{proper spectrum}
\end{eqnarray}
For convenience, we often write $\sigma_p(A_1,...,A_n)$ as $\sigma_p(A)$. If the characteristic polynomial $Q_A(z_0,z)$ is given by \eqref{char}, then 
\begin{equation}\label{charac} 
\sigma_p(A)=\{Q_A(-1,z)=0\}=\bigg\{\prod_{j=1}^s R_j(-1,z)=0\bigg\}.
\end{equation} 
For convenience, we shall often write $R_j(-1,z)$ as $R_j(z)$. Denoting the irreducible components of $\sigma_p(A)$ by 
\begin{equation}\label{component}
\Gamma_j=\{ R_j=0\},
\end{equation}
we have $\sigma_p(A)=\cup_{j=1}^s \Gamma_j$. This representation does not explicitly show the multiplicities of the spectral components, and these multiplicities play an important role in  our considerations. Hence, instead of $\sigma_p(A)$, we will often use what we call the \textit{proper projective joint spectrum in the divisor form}, which is just the zero-divisor of $Q_A(-1,z)$. We denote it by $\sigma_p^d(A)$. Then (\ref{char}) gives
\begin{equation}\label{divisor spectrum}
\sigma_p^d(A)=\displaystyle \sum_{j=1}^s t_j\Gamma_j,
\end{equation}
and we refer to it as the proper projective spectrum when there is no ambiguity.

In the subsequent discussions of this subsection, we  assume that $A$ is a Hermitian tuple.. The  projections ${\mathcal P}_j, \ j=1,..., r$, are given by \eqref{projections} (with $A$ replaced by $A_1$), and the spectral resolution of $A_1$ is 
\begin{equation}\label{spectral resolution A 1}
	A_1=\displaystyle \sum_{j=1}^r \lambda_j{\mathcal P}_j. 
\end{equation}

It is not hard to see that, for each $j$, we have $\tau_j:=(1/\lambda_j,0,\dots ,0)\in \sigma_p(A)$. Suppose that each $\tau_j$ belongs to a single component $\Gamma_i$, and it is regular with $\frac{\partial R_i(z)}{\partial z_1}\mid_{z=\tau_j} \neq 0$. This implies that in a small neighborhood ${\mathcal O}_\epsilon (\tau_j)$ we have $\sigma_p(A)\cap {\mathcal O}_\epsilon(\tau_j)=\Gamma_i\cap {\mathcal O}_\epsilon(\tau_j)$.
By the implicit function theorem $\Gamma_i$ is represented (perhaps in a smaller neighborhood of $\tau_j$) as: 
\begin{equation}\label{implicit}
z_1=z_{1,j}(z_2,...,z_{n}), \ z_{1,j}(0)=1/\lambda_j. 
\end{equation}

The following result is a special case of a more general Theorem 2.1 in \cite{S2}.

\begin{theorem}[\cite{S2}]\label{local spectral}
Let $A$ be a Hermitian tuple such that $A_1$ is  invertible. Suppose that the reciprocal of each eigenvalue of $A_1$  belongs to a single component $\Gamma_i$ and is a regular point of this component (not counting multiplicity), with  $\frac{\partial R_i}{\partial z_1}\Big|_{\tau_j} \neq 0$.
Then for each $i=1,...,r, \ j=2,...,n$ we have
\begin{equation}\label{first moment}
\mathcal{P}_i\Big(A_j+\frac{\partial z_{1,i}}{\partial z_j}(0)A_1\Big)\mathcal{P}_i=0.	
\end{equation}

\end{theorem}

\subsection{Hermitian Tuples with Characteristic Polynomial Being a Power}

Let $A$ be an admissible tuple of Hermitian invertible $k\times k$ matrices.
Assume the proper projective joint spectrum in divisor form $\sigma_p^d(A)$ is given by
\begin{equation}\label{spectrum1}
\sigma_p^d(A)=l\Big\{ R(z)=0\Big\}, 
\end{equation}
where $l>1$ and $R$ is a polynomial of degree $m$, so that $k=ml$. We assume that $R(z_1,0,...0)$ has simple roots.
Denote by $\lambda_1,...,\lambda_m$ the set of distinct eigenvalues of $A_1$. Since $A_1$ is invertible, $\lambda_j\neq 0$. Like in \eqref{projections} and \eqref{integral projection}  \  we denote by $\mathcal{P}_1,...,\mathcal{P}_m$ the  projections (which are orthogonal in this case as $A_1$ is Hermitian)  on eigenspaces of $A_1$ corresponding to $\lambda_1,...,\lambda_m$. Of course, \eqref{spectrum1} implies that all $\lambda_j$ have multiplicity $l$.

We will be dealing with the following subset of the algebra generated by $A_1,...,A_n$. 
For $q\leq m$, \  $2\leq i_1,...,,i_q\leq  n$, and a set of pairwise distinct indices $j_1,...,j_{q-1}\in \{1,...,m\}$  write 
\begin{equation}\label{W IJ}
W_{IJ} = A_{i_1}\mathcal{P}_{j_1}A_{i_2}\mathcal{P}_{j_2}\cdots A_{i_q-1}\mathcal{P}_{j_q-1}A_{i_q}.
\end{equation}
Here $I=(i_1,...,i_q), \ J=(j_1,...j_{q-1})$, and,  if $q=1$, then $J=\emptyset$. We denote by $\mathscr{W}$ the collection of all such matrices. In particular, we have $A_2,...,A_n  \in \mathscr{W}.$ 

\begin{theorem}[\cite{S2}, Theorem 3.1]\label{kipp  main}
Let $A$ be an admissible $n$-tuple of $ml\times ml$ Hermitian invertible matrices.  
Suppose $\sigma_p^d(A)$ is given by \eqref{spectrum1}, 
where $R$ is a  polynomial of degree $m$ such that $R(z_1,0,...,0)$ has simple roots as a polynomial in $z_1$. Then $A$ is unitarily equivalent to a direct sum of $l$ identical $n$-tuples of $m\times m$ matrices if and only if for each $W_{IJ}\in \mathscr{W}$ we have
\begin{eqnarray}
&\sigma_p^d(A_1,W_{IJ}) 
=l\{R_{IJ}(x_1,x_2)=0\}, \label{big spectrum}
\end{eqnarray}	
where $R_{IJ}$ is a polynomial of degree $m$.
\end{theorem}

\begin{remark} \label{explicit}\normalfont The proof of this theorem given in \cite{S2} not only establishes the reducibility of the tuple, but also explicitly constructs the corresponding reducing subspaces.
\end{remark}

\section{Tuple Extensions, Spectral Indices and Multiplicities}\label{exten}

As noted earlier, the reducibility of matrix tuples is a delicate matter, and it cannot be determined by their characteristic polynomials alone. This section addresses this issue by considering extensions of tuples. Given any tuple $A=(A_1, ..., A_n)$, tuples of the form $\widetilde{A}=(A_1, ..., A_N)$, where $N>n$ and $A_{n+1}, ..., A_N\in \mathfrak{A}(A)$, are called {\em algebraic extensions} of $A$. In the case $A$ is irreducible, Burnside's theorem shows that $\mathfrak{A}(A)=M_k(\C)$, and hence every extension of the tuple $A$ is algebraic. Throughout the remaining part of this paper, we assume $k\geq 2$. In this case, assuming $A_1, ..., A_n$ are linearly independent, there exist matrices $A_{n+1}, ... A_N$ such that 
span$\{I, A_1, ..., A_N\}=M_k(\C)$.    An arbitrary matrix $W=(w_{ij})\in M_k(\C)$ can be written as a pencil of elementary matrices: $W=\sum_{i,j=1}^kw_{ij}E_{ij}$. It is well-known that $\det W$ is an irreducible polynomial in the variables $w_{ij}$. Setting $\widetilde{A}=(A_1, ..., A_N)$, then $Q_{\widetilde{A}}(z_0,z)$ is an irreducible polynomial in $z_0, ..., z_N$ because it differs from $\det W$ by a linear change of variables. The link between the irreducibility of tuple $A$ and the irreducibility of extended characteristic polynomial $Q_{\widetilde{A}}(z_0,z)$ is the focus of this section.

\subsection{Spectral Indices and Multiplicities}

The following equivalence relation was first introduced in \cite{AY} for finite dimensional Lie algebras.

\begin{definition}
    We say that two tuples of matrices $A=(A_1, ..., A_n)$ and $B=(B_1, ..., B_n)$ are {\em spectrally equivalent} if there exists a matrix $C\in GL(n,\C)$ such that $Q_{B}(z_0,z)=Q_A(z_0, zC), z\in \C^n$.
\end{definition}

For general matrix tuples, spectral equivalence is much weaker than similarity or unitary equivalence. In particular, if the tuples $A$ and $B$  consist of strictly upper triangular matrices, then $Q_A=Q_B=z_0^k$, which yields no information about the connection between the tuples. On the other hand, if the tuples consist of bases for simple Lie algebras, then spectral equivalence implies the isomorphism of the two Lie algebras \cite{FLW,GLW,Ya24}. The following definitions capture important features of $A$ and $Q_A$.

\begin{definition}\label{s-index}
Let $A$ be a matrix tuple and assume the irreducible factorization of $Q_A$ is expressed as in \eqref{char}.

a) The {\em spectral index} of $A$ is defined as $\kappa (A)=s$.

b) The {\em spectral multiplicity} of $A$ is defined as $\nu(A)=t_1+\cdots +t_s$.
\end{definition}
\noindent Thus $\kappa(A)$ is the number of distinct irreducible factors of $Q_A$. It is clear that if tuples $A$ and $B$ are spectrally equivalent, then $\kappa(A)=\kappa(B)$ and $\nu(A)=\nu(B)$, showing that both numbers are invariant with respect to spectral equivalence. Some basic properties of $\kappa$ and $\nu$ are listed below.
\begin{itemize}
\item If $A$ is a tuple of $k\times k$ matrices, then evidently $\kappa (A)\leq \nu(A)\leq k$.
\item $\kappa(A)=k$ if and only if $Q_A$ is a product of $k$ distinct linear factors. 
\item If $A$ is a commuting tuple, then it is similar to a tuple of upper triangular matrices. Hence $Q_A$ is a product of linear factors and therefore $\nu(A)=k$.
\item $\kappa(A)=1$ if and only if $Q_A=R^d$ for some integer $d\geq 1$ and irreducible polynomial $R$.
\item $\nu(A)=1$ if and only if $Q_A$ is irreducible.
\item $\kappa(A)=\nu(A)$ if and only if $Q_A$ has no repeated factors. 
\end{itemize}

\begin{example}
If $G=\{1, g_1, ..., g_n\}$ is a finite group, then Theorem \ref{Frob} shows that \[\kappa (\lambda(\bar{g}))=|\hat{G}|,\hspace{1cm}\text{and}\hspace{1cm} \nu(\lambda(\bar{g}))=\sum_{[\pi]\in \hat{G}}d_\pi\]which is the sum of character degrees.
\end{example}

The reducibility of a general matrix tuple $A$ often cannot be derived from the factorization of $Q_A$ (Example \ref{Waterhouse}). This motivates our interest on extended tuples, and the following notions are critical.
\begin{definition}\label{lowerindex}
For a matrix $n$-tuple $A$, its {\em lower spectral index} and the {\em lower spectral multiplicity} are defined as follows:
\begin{align*}
\kappa_*(A)&=\min \{\kappa(A_1, ..., A_N)\mid N> n, \ \ A_{n+1}, ..., A_N\in \mathfrak{A}(A)\}.\\
\nu_*(A)&=\min \{\nu(A_1, ..., A_N)\mid N\geq n, \ \ A_{n+1}, ..., A_N\in \mathfrak{A}(A)\}.
\end{align*}
Further, if $A$ is Hermitian, we require that $A_{n+1}, ..., A_N$ all be Hermitian.
\end{definition}
Since $\kappa (A)=\kappa (A_1, ..., A_n, 0)$, it is obvious that $\kappa_*(A)\leq \kappa(A)$. Likewise, we have $\nu_*(A)\leq \nu(A)$. In the case $\kappa(A)=\kappa_*(A)$, we say that tuple $A$ is {\em saturated}, and in the case $\nu(A)=\nu_*(A)$, we say that $A$ is stable. For a non-saturated or non-stable tuple $A$, it would be meaningful to determine the least integers $N$ (i.e. the shortest algebraic extension of $A$) such that $\kappa(A_1, ..., A_N)=\kappa_*(A)$ and $\nu(A_1, ..., A_N)=\nu_*(A)$. There is evidence to suspect that $N=n+1$ in both cases, and a proof for $\nu_*$ will be given in Theorem \ref{nu*}. However, the case for $\kappa_*$ seems challenging, and we will post it as a conjecture.

\begin{conjecture}\label{N*}
For any tuple $A$, there exist a matrix $B\in \mathfrak{A}(A)$ such that $\kappa_*(A)=\kappa(A_1, ..., A_n, B)$. Moreover, if $A$ is Hermitian, then $B$ can be chosen to be Hermitian.
 \end{conjecture}

We will see in the next subsection that the conjecture holds for irreducible tuples as well as tuples whose characteristic polynomials have no repeated factors.

\subsection{Irreducible Tuples} 

Let $\mathcal{R}$ be a unique factorization domain, $\mathcal{K}$ be its fraction field, and consider a polynomial
\[f(x)=x^d+a_{d-1}x^{d-1}+\cdots+a_1x+a_0\in \mathcal{R}[x], \qquad d\geq 1.\]
Eisenstein's criterion (see \cite{CaF}, p.22-23 ) asserts that if $\alpha\in {\mathcal R}$ is a prime element such that 
\[\alpha \mid a_j\quad(0\leq j<d), \qquad \alpha^2\nmid a_0,\]
then $f$ is irreducible in both $\mathcal{R}[x]$ and $\mathcal{K}[x]$. This criterion leads to the following lemma regarding tuple extension and irreducibility.

\begin{lemma}
\label{nilpotent-extension}
For any nonscalar matrix $T\in M_k(\C)$, there exists a
matrix $B\in M_k(\C)$ such that
\begin{equation}\label{TB}
\det(xI+yT+tB)
\end{equation}
is irreducible in $\C[x,y,t]$. Moreover, if $T$ is Hermitian, then $B$ can be chosen to be Hermitian.
\end{lemma}

\begin{proof}
Since $T$ is nonscalar, there exists $u\in\C^k$ such that $u$ and $Tu$
are linearly independent. Choose a basis
$v_1,\ldots,v_k$ with \[v_1=u,\qquad v_k=Tu,\]
and let $P$ be the matrix whose columns are these basis vectors. Set
$S=P^{-1}TP=(s_{ij})$. The equality $Tv_1=v_k$ implies
\begin{equation}
 \label{eq:corner}
 s_{k1}=1.
\end{equation}

Let $E_{ij}$ denote the standard elementary matrices and put
\[J=\sum_{i=1}^{k-1}E_{i,i+1}, \qquad B=PJP^{-1}.\]
Thus $J$ is the Jordan matrix of size $k$, and $J^{k-1}=E_{1k}\ne 0$. Similarity invariance of the
determinant gives
\[\det(xI+yT+tB)=\det(xI+yS+tJ)=:F(x,y,t).\]
We apply Eisenstein's criterion to $F$ as a polynomial in $x$
over the polynomial ring $\C[y,t]$, using the prime element $\alpha=y$. Notice that
$\C[y,t]$ is a unique factorization domain and $\C[y,t]/(y)\cong\C[t]$ is an integral domain.

Write
\[ F(x,y,t)=x^k+\sum_{j=0}^{k-1}a_j(y,t)x^j.\] Since $J$ is nilpotent,
\[ F(x,0,t)=\det(xI_k+tJ)=x^k.\]
Consequently, we have $y\mid a_j(y,t)$ for every $0\leq j<k$.

It remains to show that $y^2\nmid a_0(y,t)$. Indeed, an inspection of the entries of $tJ+yS$ and expanding $\det(tJ+yS)$ along the first column by the cofactor theorem, we obtain
\[a_0(y,t)=(-1)^{k-1}s_{k1}yt^{k-1}+O(y^2)=(-1)^{k-1}yt^{k-1}+O(y^2),\]where $O(y^2)$ stands for the sum of the terms with a factor $y^2$. Thus $a_0$ is divisible by $y$ but not by $y^2$. Eisenstein's criterion then ensures that $F$ is irreducible in $\C[y,t][x]=\C[x,y,t]$.

To obtain a Hermitian $B$ when $T$ is Hermitian, we let $O_T$ be the set of all $B\in M_k(\C)$ for which \eqref{TB} is irreducible. We have shown that $O_T\ne\varnothing$. We claim that it is Zariski open. Indeed, let $H_d$ be the set of homogeneous polynomials in the form \eqref{expandQ} of degree $d$ in the complex variables $x,y, t$. The reducible locus in $H_k$ is
 \[\mathcal{S}:=\bigcup_{d=1}^{k-1} H_d\cdot H_{k-d},\] which is Zariski closed in $H_k$.
 Since $\rho: B\to \det (xI+yT+tB)$ is a polynomial map from $M_k(\C)$ into $H_k$, the pre-image
 $O_T=\rho^{-1}(H_k\setminus {\mathcal S})$ is Zariski open in $M_k(\C)$.  

Finally, since the set of Hermitian matrices is Zariski dense in $M_k(\C)$, the open set $O_T$ must contain a Hermitian matrix, completing the proof of the lemma.
\end{proof}

Recall that Burnside's theorem asserts that a matrix tuple $A$ is irreducible if and only if ${\mathfrak A}(A)=M_k(\C)$. Combined with Lemma \ref{nilpotent-extension}, we arrive at the following fact.

\begin{theorem}\label{mainthm}
 Let $k>1$. A tuple $A=(A_1, ..., A_n)$ of $k\times k$ matrices is irreducible if and only if there exists an algebraic extension $\widetilde{A}=(A_1, ..., A_n, B)$ such that $Q_{\widetilde{A}}$ is irreducible. 
\end{theorem}
\begin{proof}
If $A$ is irreducible, then at least one of the matrices, say $A_1$, is nonscalar. Since ${\mathfrak A}(A)=M_k(\C)$ in this case, Lemma \ref{nilpotent-extension} proves the existence of a matrix $B\in \mathfrak{A}(A)$ such that $\det (z_0I+z_1A_1+z_{n+1}B)$ is irreducible. This forces 
$Q_{\widetilde{A}}$ to be irreducible.

On the other hand, if there exists an algebraic extension $\widetilde{A}=(A_1, ..., A_n, B)$ such that $Q_{\widetilde{A}}$ is irreducible, then obviously the pencil $z_0I+z_1A_1+\cdots+z_nA_n+z_{n+1}B$ does not admit a simultaneous block triangularization of the coefficient matrices, and hence it is irreducible. Since $B\in \mathfrak{A}(A)$, tuple $A$ must be irreducible.
\end{proof}

For convenience, we set $\Alg=\Alg(A)$, $V=\C^k$, and view $V$ as a left $\Alg$-module. Fix a composition series
\begin{equation}
 0=V_0\subsetneq V_1\subsetneq\cdots\subsetneq V_{\ell}=V,
 \qquad V_j/V_{j-1}\ \text{simple}.\label{eq:series}
\end{equation}
The Jordan--H\"older theorem shows that the composition length $\ell_\Alg$ and
the multiset of simple quotient modules $V_j/V_{j-1}$ are independent of the chosen series \cite[Theorem 3.7.1]{EGHLSVY}. 
Let $S_1,\ldots,S_s$ be representatives of its distinct simple-module types,
with multiplicities $m_1,\ldots,m_s$. Write
\[d_i=\dim_{\C}S_i,\qquad \rho_i:\Alg\longrightarrow\End_{\C}(S_i).\]
Then\[\ell_\Alg=\sum_{i=1}^{s}m_i,\qquad k=\sum_{i=1}^{s}m_i d_i.\]

We recall the following lemma in \cite[Theorem 3.2.2(ii)]{EGHLSVY}.
\begin{lemma}\label{lem:density}
Let $\Alg$ be a unital complex algebra and let $S_1,\ldots,S_s$ be pairwise
non-isomorphic finite-dimensional simple $\Alg$-modules. Then
\[
 \rho=(\rho_1,\ldots,\rho_s):
 \Alg\longrightarrow\bigoplus_{i=1}^{s}\End_{\C}(S_i)\label{eq:density}
\]
is surjective.
\end{lemma}

For a nonzero polynomial $R$, we let $\nu(R)$ be the number of its irreducible
factors counted with multiplicity, and put $\nu(c)=0$ for nonzero constants.
Unique factorization gives
\begin{equation}
 \nu(R_1R_2)=\nu(R_1)+\nu(R_2),\qquad \nu(R^m)=m\nu(R).\label{eq:additivity}
\end{equation}
These identities hold even when $R_1$ and $R_2$ have common factors.
In particular, for a matrix tuple $A$, we have $\nu(A)=\nu(Q_A)$.

\begin{theorem}\label{nu*}
For any tuple $A$, we have 

a) $\nu_*(A)=\ell_{\Alg}$;

b) there exists a matrix $B\in \mathfrak{A}$ such that $\nu_*(A)=\nu(A_1, ..., A_n, B)$;

c) if $A$ is Hermitian, then $B$ can be chosen to be Hermitian.
\end{theorem}

\begin{proof}
We first prove that every algebraic extension of $A$ has spectral
multiplicity at least $\ell$. Let
\[\widetilde A=(A_1,\ldots,A_n,C_1,\ldots,C_r), \qquad C_b\in\Alg.\]
Because every $V_i$ is an $\Alg$-submodule, all matrices in
$\widetilde A$ preserve the entire composition series. In a basis adapted
to this series, these matrices are therefore simultaneously block upper
triangular. Their diagonal blocks represent their actions on the
quotients $V_i/V_{i-1}$.

For each $i$, define
\[F_i(z_0,z,w)=\det\left(z_0I_{d_i}+\sum_{a=1}^n z_a\rho_i(A_a)
 +\sum_{b=1}^r w_b\rho_i(C_b)\right).\]
The actions on isomorphic simple quotients are simultaneously similar
and hence have the same determinant polynomial. Taking the determinant
of the block upper triangular pencil gives
\[Q_{\widetilde A}(z_0,z,w)=\prod_{i=1}^s F_i(z_0,z,w)^{m_i}.\]
Each $F_i$ is monic in $z_0$ of degree $d_i\geq1$. Thus it is
nonconstant, and $\nu(F_i)\geq1$. By the additivity identities \eqref{eq:additivity},
\[\nu(\widetilde A)=\sum_{i=1}^s m_i\nu(F_i)\geq\sum_{i=1}^s m_i=\ell_{\Alg}.\]
Since this applies to every algebraic extension,
\begin{equation}
 \nu_*(A)\geq\ell_{\Alg}.
 \label{eq:thm79-lower-bound}
\end{equation}
The same lower bound holds when the allowed extensions are restricted
to Hermitian matrices.

We next construct a single matrix attaining this bound. For each $i$,
consider the tuple
\[A^{(i)}=\bigl(\rho_i(A_1),\ldots,\rho_i(A_n)\bigr)\]
acting on $S_i$. This tuple is irreducible. Indeed, a subspace invariant
under its entries is invariant under their unital generated algebra
$\rho_i(\Alg)$, and hence is an $\Alg$-submodule of the simple module
$S_i$.

By Theorem \ref{mainthm}, there exists $B_i\in\End_{\C}(S_i)$ such that
\begin{equation}\label{eq:PI}
P_i(z_0,z,t)=\det\left(z_0I_{d_i}+\sum_{a=1}^n z_a\rho_i(A_a)+tB_i\right)
\end{equation}
is irreducible. When $d_i=1$, this conclusion holds directly with
$B_i=0$, since $P_i$ is a nonconstant linear polynomial.

Since $S_1,\ldots,S_s$ are pairwise nonisomorphic, Lemma \ref{lem:density} shows that
\[\rho=(\rho_1,\ldots,\rho_s):\Alg\longrightarrow \bigoplus_{i=1}^s\End_{\C}(S_i)\]
is surjective. Consequently, there exists one matrix $B\in\Alg$
satisfying
\[ \rho_i(B)=B_i,\qquad 1\leq i\leq s.\]
Applying the preceding determinant factorization to the tuple extension $(A,B)$ gives
\begin{equation}\label{QAB}
Q_{(A,B)}(z_0,z,t) =\prod_{i=1}^sP_i(z_0,z,t)^{m_i}.
\end{equation}
Each $P_i$ is irreducible, so equation \eqref{eq:additivity} yields
\begin{equation}
 \nu(A,B)
 =\sum_{i=1}^s m_i\nu(P_i)
 =\sum_{i=1}^s m_i
 =\ell_\Alg.
 \label{eq:thm79-attainment}
\end{equation}
The polynomials $P_i$ need not be pairwise distinct: when two coincide,
their exponents add, leaving the total spectral multiplicity unchanged.
Together, \eqref{eq:thm79-lower-bound} and
\eqref{eq:thm79-attainment} prove
\[\nu_*(A)=\ell_{\Alg}=\nu(A,B),\]
which establishes parts a) and b) of the theorem.

It remains to obtain a Hermitian $B$ when $A$ is Hermitian. In this case
$\Alg$ is a $C^*$-algebra, each quotient $V_j/V_{j-1}$ can be identified with $V_j\ominus V_{j-1}$, which is a reducing subspace for $\Alg$, and the corresponding block tuple $A^{(i)}$ is Hermitian and irreducible. Moreover, each $\rho_i$ is now a $^*$-homomorphism. Thus, for each $i$ we may choose Hermitian $B_i$ according to Lemma \ref{nilpotent-extension} such that $P_i(z_0, z, t)$ in \eqref{eq:PI} is irreducible. For $d_i=1$, we take the Hermitian matrix $B_i=0$.

Apply Lemma \ref{lem:density} to obtain a simultaneous lift $C\in\Alg$ with
\[ \rho_i(C)=B_i,\qquad 1\leq i\leq s,\]
and set
\[B=\frac{C+C^*}{2}.\]
Then $B\in\Alg$ is Hermitian, and
\[\rho_i(B)=\frac{\rho_i(C)+\rho_i(C)^*}{2}=B_i.\]
Thus the identities \eqref{QAB} and
\eqref{eq:thm79-attainment} remain valid for this Hermitian $B$.
Together with the lower bound for Hermitian extensions, this proves
all three assertions.
\end{proof}

Recall that a matrix tuple $A$ has the property $\kappa(A)=\nu(A)$ precisely when $Q_A$ has no repeated factors. The following result describes additional settings where a single element algebraic extension has lower spectral and multiplicity indices. 
\begin{corollary}\label{no repeated}
Suppose that the characteristic polynomial $Q_A$ has no repeated factors. Then $\kappa_*(A)=\nu_*(A)$ and there exists an algebraic extension $\widetilde{A}=(A_1,...,A_n,B)$ such that \[\kappa (\widetilde{A})=\nu(\widetilde{A})=\kappa_*(A)=\nu_*(A).\]	
\end{corollary}
\begin{proof}
Let \(\widetilde A\) be any algebraic extension of \(A\). Normalize its irreducible determinant factors to be monic in \(z_0\). Every nonconstant factor has positive \(z_0\)-degree, and retains that degree when the new variables are set equal to zero. If \(Q_{\widetilde A}\) had a repeated factor, its specialization \(Q_A\) would have a repeated nonconstant factor. This contradicts the hypothesis.

Thus every algebraic extension has square-free characteristic polynomial, and consequently
$\kappa(\widetilde A)=\nu(\widetilde A)$
for every such extension. Taking minima gives $\kappa_*(A)=\nu_*(A)$.
By Theorem \ref{nu*}, there exists a single \(B\in\mathfrak A(A)\), Hermitian when required, such that
$\nu(A,B)=\nu_*(A)$.
The extension is square-free, so
\[\kappa(A,B)=\nu(A,B)=\kappa_*(A)=\nu_*(A).\]
\end{proof}

\subsection{Related Questions}

In light of Theorem \ref{mainthm} and Corollary \ref{no repeated}, we are tempted to raise the following problems related to Conjecture \ref{N*}. 

\begin{problem}
Consider a matrix tuple $A$.
\begin{itemize}
\item If the answer to the conjecture is negative, how to determine the minimal length of the corresponding extension of $A$?
\item Can we always choose an extension that minimizes the spectral index and multiplicity simultaneously?
\end{itemize}
\end{problem}

We note that a construction of $B$ that appeared in Theorem \ref{mainthm} and Corollary \ref{no repeated} can be challenging. So, in practice, it might be rather difficult to make a decision about the reducibility of a tuple by finding this extension matrix $B$. Fortunately, for Hermitian tuples, the issue of reducibility can be resolved in a different approach. In Section \ref{Kippenhahn}, we will present an algorithm for constructing algebraic extensions for this purpose. This algorithm leads to a necessary and sufficient condition for the validity of Kippenhahn's conjecture. 

Furthermore, it is meaningful to explore applications of Theorem \ref{mainthm} to representation theory. A finite group $G=\langle g_1, ..., g_n\rangle$ is a basis for its group algebra $\C[G]$. 
 If $\pi$ is a $k$-dimensional representation, then $\pi(G)$ spans $\pi(\C[G])$. It is obvious that $\mathfrak{A}(\pi(\bar{g}))=\pi(\C[G])$. In the case $\pi$ is an irreducible representation of dimension $k$, Burnside's theorem implies that $\pi(\C[G])=M_k(\C)$, and hence there are elements $g_{n+1}, ..., g_{N}\in G$ such that the characteristic polynomial of the tuple $(\pi(g_1), ..., \pi(g_{N}))$ is irreducible, i.e., $\kappa_*(\pi(\bar{g}))=1$. This is a part of Frobenius' theorem (Theorem \ref{Frob}). However, the characteristic polynomial of $(\pi(g_1), ..., \pi(g_n))$ is not necessarily irreducible, as shown in \cite{KV} for the alternating group $A_6$. 
 
\begin{problem}\label{grp-extension}
Let $\pi$ be a finite dimensional irreducible representation of $G$.

a) Must there be an element $g_{n+1}\in G$ such that the characteristic polynomial of $(\pi(g_1), ..., \pi(g_{n+1}))$ is irreducible?

b) If the answer to a) is yes, what is the minimal word length of such $g_{n+1}$?
\end{problem}

\vspace{.2cm}

\section{The Characteristic Graph of Matrix Tuples}\label{characteristic graph}

This section presents a geometric criterion for the irreducibility of matrix tuples. The approach here relates the common invariant subspace problem to graph theory. We associate with a tuple some specific digraphs, whose strong connectedness is related to a tuple's reducibility. 


\subsection{Strongly Connected Graphs}

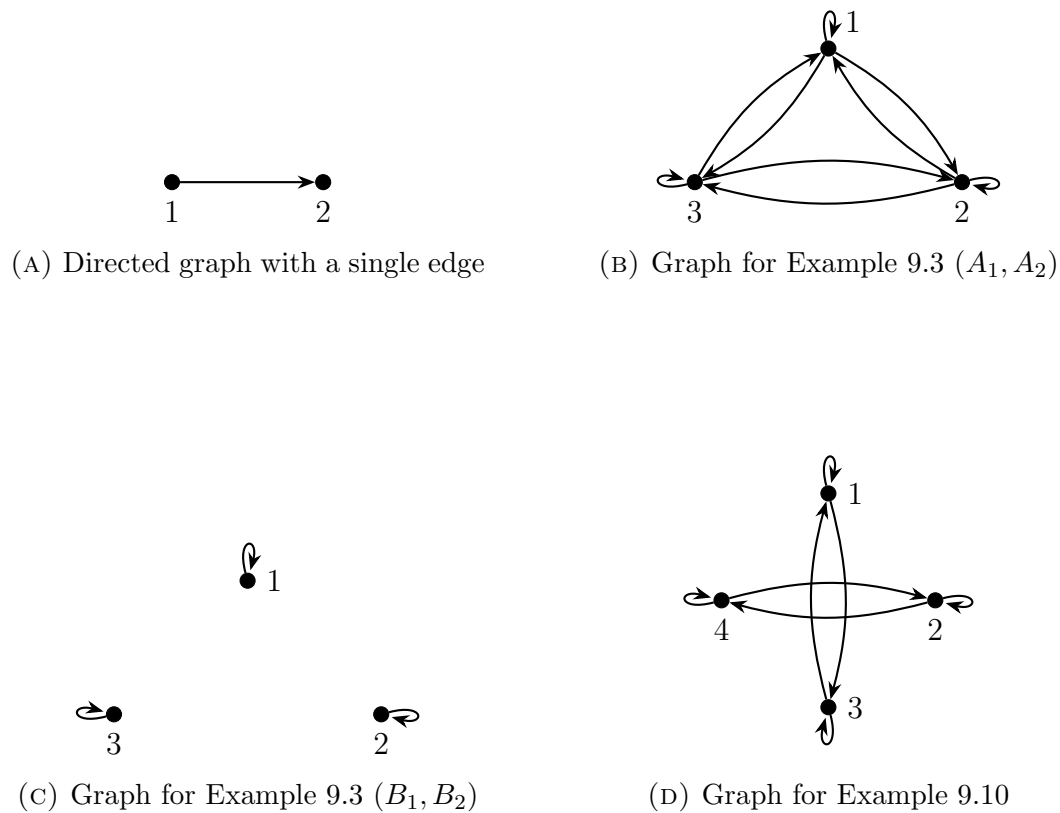
\begin{figure}[htbp]
    \centering
    \begin{subfigure}[b]{0.45\textwidth}
        \centering
        \begin{tikzpicture}[>=Stealth, node distance=2cm, main/.style = {draw, circle, fill=black, inner sep=2pt}]
            \node[main, label=below:1] (1) {};
            \node[main, label=below:2] (2) [right of=1] {};
            
            \draw[->, thick] (1) -- (2);
        \end{tikzpicture}
        \caption{Directed graph with a single edge}
        \label{fig:graph_a}
    \end{subfigure}
    \hfill
    \begin{subfigure}[b]{0.45\textwidth}
        \centering
        \begin{tikzpicture}[>=Stealth, node distance=2.5cm, main/.style = {draw, circle, fill=black, inner sep=2pt}]
            \node[main, label=above right:1] (1) {};
            \node[main, label=below:2] (2) [below right of=1] {};
            \node[main, label=below:3] (3) [below left of=1] {};
            
            \path[->, thick]
                (1) edge [bend left=15] node {} (2)
                (2) edge [bend left=15] node {} (1)
                (2) edge [bend left=15] node {} (3)
                (3) edge [bend left=15] node {} (2)
                (3) edge [bend left=15] node {} (1)
                (1) edge [bend left=15] node {} (3);
        \end{tikzpicture}
        \caption{Graph for Example \ref{example8.7} $(A_1, A_2)$}
        \label{fig:graph_b}
    \end{subfigure}

    \vspace{2cm}

    \begin{subfigure}[b]{0.45\textwidth}
        \centering
        \begin{tikzpicture}[>=Stealth, node distance=2.5cm, main/.style = {draw, circle, fill=black, inner sep=2pt}]
            \node[main, label=right:1] (1) {};
            \node[main, label=below:2] (2) [below right of=1] {};
            \node[main, label=below:3] (3) [below left of=1] {};
            
        \end{tikzpicture}
        \caption{Graph for Example \ref{example8.7} $(B_1, B_2)$}
        \label{fig:graph_c}
    \end{subfigure}
    \hfill
    \begin{subfigure}[b]{0.45\textwidth}
        \centering
        \begin{tikzpicture}[>=Stealth, node distance=2cm, main/.style = {draw, circle, fill=black, inner sep=2pt}]
            \node[main, label=right:1] (1) {};
            \node[main, label=below:2] (2) [below right of=1] {};
            \node[main, label=right:3] (3) [below left of=2] {};
            \node[main, label=below:4] (4) [above left of=3] {};
            
            \path[->, thick]
                (1) edge [bend left=15] node {} (3)
                (3) edge [bend left=15] node {} (1)
                (4) edge [bend left=15] node {} (2)
                (2) edge [bend left=15] node {} (4);
        \end{tikzpicture}
        \caption{Graph for Example \ref{example8.16}}
        \label{fig:graph_d}
    \end{subfigure}
    \caption{Characteristic graphs}
\end{figure}

\vspace{2mm}

A {\em directed graph}, also called {\em digraph}, consists of a finite number of vertices $\{1,...,k\}$ and  edges connecting pairs of vertices in a fixed direction. An edge from vertex $i$ to  vertex $j$ is denoted by $\overrightarrow{ij}$. By definition all trivial edges $\overrightarrow{ii}$ are included in a digraph. Given a directed graph $\mathcal{G}$, a path 
$$i_1\to i_2 \to \cdots \to i_p$$
is a $\mathcal{G}$-path, if $\overrightarrow{i_ri}_{r+1}\in \mathcal{G}$ for all $r=1,...,p-1$.
Two vertices $i$ and $j$ are said to be strongly $\mathcal{G}$-connected, if there are $\mathcal{G}$-paths
$$ i=i_1\to i_2 \to \cdots \to i_p=j \ \mbox{and} \ j=j_1\to j_2 \to \cdots \to j_s=i.$$
 A subset $S\subset \{1,...,.k\}$ is {\em strongly $\mathcal{G}$-connected}, if every pair of vertices in $S$ is strongly $\mathcal{G}$-connected. If the whole set of vertices $\{1,...,k\}$ are strongly $\mathcal{G}$-connected, the graph $\mathcal{G}$ is said to be strongly connected. Of course, every single vertex subset is strongly connected, but not every directed graph has strongly connected subsets that include more than one vertex. For example, the graph with 2 vertices $\{1,2\}$ and a single non-trivial edge $\overrightarrow{12}$ has no non-trivial strongly connected subsets (Fig. 2 (A)). 
 
 We also mention that every strongly connected subset is included in a unique maximal strongly connected subset, which is called a strongly connected component. Of course, any two different strongly connected components are disjoint. 
 
 \subsection{Totally Strongly Disconnected Digraphs}
We start this subsection with recalling the following definitions.

\begin{definition}
A vertex $i$ of a digraph $\mathcal{G}$ is called a {\em  source}, if there are no directed edges coming to $i$, that is $\overrightarrow{ji}\notin \mathcal{G}$ for all $j\neq i$. A vertex $i$ is called a {\em sink}, if there are no directed edges emanating from $i$, that is $\overrightarrow{ij} \notin \mathcal{G}$, for all $j\neq i$.	
\end{definition}

\begin{definition}
Given a digraph $\mathcal{G}$ the {\em dual digraph $\mathcal{G}^*$} is the digraph with the same vertices $\{1,...,k\}$ and the edges in opposite directions, that is $\overrightarrow{ij}\in \mathcal{G}^* \iff \overrightarrow{ji}\in \mathcal{G}$.
\end{definition}

It is clear that a vertex $i$ is a source in $\mathcal{G}$ if and only if it is a sink in $\mathcal{G}^*$, and vice versa, $I$ is a sink in $\mathcal{G}$ if and only if it is a source in $\mathcal{G}^*$
\begin{definition}
A digraph $\mathcal{G}$ is {\em totally strongly disconnected} if every strongly connected component of $G$ consists of a single vertex.
\end{definition}

The following fact is well-known. We present a proof here to make our presentation self-contained.

\begin{proposition}\label{sink or source}
Every totally strongly disconnected digraph $\mathcal{G}$ has a sink and a source.	
\end{proposition}

\begin{proof}
Since $\mathcal{G}$ is totally strongly disconnected $i\neq j, \ \overrightarrow{ij}\in \mathcal{G},  \  \Longrightarrow \overrightarrow{ji}\notin \mathcal{G}$. Now, if the vertex $1$ has no nontrivial incoming and outgoing edges, then it is both a sink and a source, and we are done. If $1$ has only nontrivial incoming edges, it is a sink. Suppose that 1 has a nontrivial outgoing edge. Without loss of generality we may assume that this edge is $\overrightarrow{12}$. Then $\overrightarrow{21}\notin \mathcal{G}$. If there are no nontrivial outgoing edges emanating from 2, it is a sink, if there is - without loss of generality we may assume that $\overrightarrow{23}\in \mathcal{G}$. Since $\mathcal{G}$ is totally strongly disconnected, neither $\overrightarrow{31}$ nor $\overrightarrow{32}$ belongs to $\mathcal{G}$. Continue this way. Either we come to a vertex $i$ which is a  sink, or we will see that 
$$ 1\to 2\to \cdots \to k$$
is a $\mathcal{G}$-path, but there are no edges $\overrightarrow{ij}$ in $\mathcal{G}$ with $i>j$. This implies that $k$ is a sink, and we proved that every totally strongly disconnected digraph has a sink.

Further, $\mathcal{G}$ is totally strongly disconnected if and only if $\mathcal{G}^*$ is. By what we have just shown, if $\mathcal{G}$ is totally strongly disconnected, $\mathcal{G}^*$ has a sink, which is a source of $\mathcal{G}$. We are done.	
\end{proof}

 \subsection{Characteristic Graphs}

Consider a tuple $A=(A_1,...,A_n)$, where $A_m=(a_{ij}^m)_{i,j=1}^k$. 
\begin{definition}\label{graph}
The {\em characteristic graph} of $A$, denoted by $\mathcal{G}_{A}$, is defined as follows:
\begin{itemize}
\item the vertices of $\mathcal{G}_{A}$ are the set $\{1,...,k\}$;
\item for $i\neq j$ a directed edge $\overrightarrow{ij}$ is present in $\mathcal{G}_{A}$ if and only if  there exists $1\leq m\leq n$ such that $a_{ij}^m\neq 0$.	
\end{itemize}
\end{definition}

\begin{remark}\label{ext}
\normalfont It follows from the Definition \ref{graph} that, if $A\subset \widetilde{A}$ and $\mathcal{G}_A$ is strongly connected, then $\mathcal{G}_{\widetilde{A}}$ is strongly connected as well.
\end{remark}

In general, the characteristic graph might be different for similar matrices. 
\begin{example}\label{example8.7}
\normalfont  Consider matrices
$$A_1=I_3, \ \ A_2=\left ( \begin{array}{ccc} 1&1&1\\1&1&1 \\1&1 &1 \end{array}\right ). $$ 
The characteristic graph $\mathcal{G}_A$ has 3 vertices, and every edge $\overrightarrow{ij}, \ i,j=1,2,3$ is in it. Thus, $\mathcal{G}_A$ is strongly connected (Fig. 2 (B)).

Since $A_2$ is similar to $B_2$ below, tuple $A$ is similar to tuple $B=(B_1,B_2)$,
$$ B_1=I_3, \ B_2=\left ( \begin{array}{ccc} 3&0&0\\0&0&0 \\0&0 &0 \end{array}\right ).$$ 
But $\mathcal{G}_B$ contains only edges $\overrightarrow{ii}$ (Fig. 2 (C)) and thus is not strongly connected.
\end{example}

The following proposition is straightforward.

\begin{proposition}\label{transform invariant}
Let $A=(A_1, ..., A_n)$ be a tuple of $k\times k$ matrices. If $C\in GL(n,\C)$ 
, then $\mathcal{G}_{A}=\mathcal{G}_{CA}$.
\end{proposition}
\begin{proof} For $i\neq j$ let $\widehat{a}_{ij}=(a_{ij}^1,...,a_{ij}^n)^T\in \C^n$.  If $\overrightarrow{ij}\in \mathcal{G}_{A}$, then $\widehat{a}_{ij}\neq 0$, and, hence, $C\widehat{a}_{ij}\neq 0$. This means that the $(ij)$th entry of at least 1 of $(CA)_1,...,(CA)_n$ does not vanish, so that $\overrightarrow{ij}\in \mathcal{G}_{CA}$.

If $\overrightarrow{ij}\notin \mathcal{G}_A$, then $a_{ij}^m=0$ for all $m=1,...,n$, and, therefore, $C\widehat{a}_{ij}=0$ so that $\overrightarrow{ij}\notin \mathcal{G}_{CA}$.
\end{proof}




 \begin{proposition}\label{eigenspace}
 Let $A$ be tuple of $k\times k$ matrices. If a vertex $i$ is a source for $\mathcal{G_A}$, then $e_i$ is a common eigenvector of the matrices in $A$. If a vertex $i$ is a sink for $\mathcal{G_A}$, then span$\{e_j: \ j\neq i\}$ is a common invariant subspace for $A$.	
 \end{proposition}
\begin{proof}
If $i$ is a source for $\mathcal{G}_A$, then $\overrightarrow{ji}\notin \mathcal{G}_A$ for all $j\neq i$. This means that for every matrix in $A$ the $(j,i)$th entry is 0, so that $e_i$ is an eigenvector for each matrix in $A$. Similarly, if $i$ is a sink for $\mathcal{G}_A$, then all $(i,j)$th entries of matrices in $A$ vanish for $i\neq j$, and this implies that span$\{e_j: \ j\neq i\}$ is invariant under the action of matrices in $A$.
	\end{proof}

Our next result unveils a somewhat surprising link between the characteristic graph $\mathcal{G}_A$ and the reducibility of $A$

\begin{theorem}\label{not connected}
A matrix tuple $A$ is reducible if and only if there is a tuple $B$ similar to $A$ such that $\mathcal{G}_B$ is not strongly connected.
\end{theorem}

\begin{proof}

Suppose that tuple $A$ has a common invariant subspace $L$. Choose a basis $\{ e_1,...,e_r\}$  for $L$ and a basis $\{e_{r+1},...,e_k\}$ for $L^\perp$, and let $B_m=(b^m_{ij})$ be the matrix of the transformation of $\C^k$ given by $A_m$ written in the basis $\{e_j: j=1,...,k\}$. Then the tuple $B$ is similar to $A$. Since  $L$ is invariant, for every $j\geq r+1, \ i\leq r$ we have $b_{ji}^m=0$ for all $m=1,...,n$. This, evidently, implies that there is no $\mathcal{G}_B$-path from any $j\geq r+1$ to any $i\leq r$. Hence $\mathcal{G}_B$ is not strongly connected, which proves the ``only if" direction.

Suppose $B$ is similar to $A$ and $\mathcal{G}_B$ is not strongly connected. The set $\{1,...,k\}$ of vertices of $\mathcal{G}_B$ is the union of strongly connected components:
$$\{1,...,k\}=\cup_{m=1}^p S_m. $$
Since $\mathcal{G}_B$ is not strongly connected $p>1$.
Consider  the following digraph $\widetilde{\mathcal{G}}_B:$
\begin{itemize}
\item $\widetilde{\mathcal{G}}_B$ has $\{1,...,p\}$ as vertices;
\item $\overrightarrow{ij}\in \widetilde{\mathcal{G}}_B$ if and only if there are $l\in S_i$ and $q\in S_j$ such that $\overrightarrow{lq}\in \mathcal{G}_B$.
\end{itemize}
We claim that $\widetilde{G}_B$ is totally strongly disconnected. Indeed, suppose that there are $1\leq i,j\leq p, \ i\neq j$ such that $i$ and $j$ are $\widetilde{\mathcal{G}}_B$-strongly connected. This means that there are $\widetilde{\mathcal{G}}_B$-pathes
$$i=i_1\to i_2\to \cdots \to i_s=j, \ \mbox{and} \ j=i_s\to i_{s+1} \to \cdots \to i_N=i, $$
but that would imply that 
$$ \widehat{S}=\cup_{l=1}^{N-1} S_{i_l}$$
is strongly connected in $\mathcal{G}_B$, which contradicts maximality of the components.

Now, Proposition \ref{sink or source} implies that $\widetilde{\mathcal{G}}_B$ has a source. Let $i$ be a source for $\widetilde{\mathcal{G}}_B$. We claim that
$$L=\SP \{e_j: \ j\in S_i\} $$
is a common invariant subspace for matrices in $B$. To see this we observe that, since the components are disjoint, for every $t\in S_i$ and $s\notin S_i$ we have $\overrightarrow{st}\notin \mathcal{G}_B$, as otherwise, if $s\in S_j$, the edge $\overrightarrow{ji}\in \widetilde{\mathcal{G}}_B$, contradicting the fact that $i$ is a source. This means that. for every matrix in $B$ the $(st)$th entry vanishes, and, therefore, $L$ is a proper common invariant subspace for matrices in $B$. Since $A\sim B$, $A$ has a common invariant subspace as well.

\end{proof}

In general, it might be not easy to find a similar tuple with not strongly connected characteristic graph (or to prove that there is no such tuple). Our next result describes an important case when such identification is handy.

\vspace{.1cm}

First, since admissible transformations of matrix tuples always exist (Theorem \ref{adm}) and they do not change the lattice of common invariant subspaces, in the subsequent discussions we assume without loss of generality that $A$ is an admissible tuple. 

Let $D\in GL(k)$. In  Remark \ref{sim} we denoted by $A^D$ the following tuple similar to $A$:
\[A^D=(A_1^D, ..., A_n^D):=(DA_1D^{-1}, ..., DA_nD^{-1})\]
and saw that that $A$ is admissible, if and only if $A^D$ is admissible as well. In our considerations below we will be interested in tuples with $A_1$ being a diagonalizable matrix.  One particular case when this happens is when $A$ is admissible and the characteristic polynomial $Q_A$ has no repeated factors.. 

\begin{proposition}\label{diagonalizable}
Let $A$ be a tuple of $k\times k$ matrices. Suppose that each eigenvalue of $A_1$ has multiplicity 1, so that $A_1$ is diagonalizable. Then, if $\mathcal{G}_{A^D}$ is strongly connected for some matrix $D$ that diagonalizes $A_1$, then the same is true for every such matrix. 	
\end{proposition}

\begin{proof}
Suppose that  $A^D_1=\Lambda:=\diag(\lambda_1,...,\lambda_k)$, then for any other matrix $\widetilde{D}\in GL(k)$
such that $A^{\widetilde{D}_1}=\Lambda$, the matrix $D\widetilde{D}^{-1}$ must commute with $\Lambda$. Since $\lambda_i\neq \lambda_j$ for all $i\neq j$, $D\widetilde{D}^{-1}$ must be diagonal with non-trivial diagonal entries, as it is invertible. Hence, $(\widetilde{D}A_m\widetilde{D}^{-1})_{ij}\neq 0 \iff (DA_mD^{-1})_{ij} \neq 0$, so the characteristic graphs are the same. If we do not fix the order of diagonal entries in $\Lambda$, we have to add a possible permutation of basic vectors, which also does not affect the strong connectedness of the graph.   
	\end{proof}

\begin{theorem}\label{graph connected}
Let $A$ be an admissible tuple of $k\times k$ matrices. Suppose the characteristic polynomial $Q_A$ has no repeated factors. Then  $A$ is irreducible if and only if $\mathcal{G}_{A^D}$ is strongly connected for some (and, therefore, for every) matrix $D$ that diagonalizes $A_1$.	
\end{theorem}
\begin{proof}
For the ``only if" part of the theorem, if $\mathcal{G}_{A^D}$ is not strongly connected, then, as $A$ is similar to $A^D$, by Theorem \ref{not connected} the tuple $A$ must be reducible.

Conversely, suppose that $Q_A$ has no repeated factors, and for some $D$ the digraph $\mathcal{G}_{A^D}$ is strongly connected. 
Since $Q_A$ has no repeated factors and since $A^D$ is admissible,  the matrix $A_1$ has $k$ distinct eigenvalues $\lambda_1,...,\lambda_k$, and let $e_1,...,e_k$ be an eigenbasis for $A_1$ such that $A_1^D=diag(\lambda_1,...,\lambda_k)$.
 Then each projection $\mathcal{P}_i$ is of rank 1. Denote by $E_{ij}$ the elementary matrices in the basis $\{e_1,...,e_k\}$, that is
$$ (E_{ij})_{ml}=\left \{ \begin{array}{cc} 1 & i=m, \  j=l \\ 0 & \mbox{otherwise}. \end{array}\right.$$
 Then for all $1\leq i,j \leq k$ and $1\leq m\leq n$, we have $$\mathcal{P}_i A^D_m\mathcal{P}_j=\widetilde{a}_{ij}^m E_{ij}, $$
where $A^D_m=(\widetilde{a}^m_{ij})_{i,j=1}^k$. Proposition \ref{spectral projections} shows that if $\overrightarrow{ij}\in \mathcal{G}_{A^D}$ then $E_{ij}\in \mathfrak{A}(A^D)$. Since $\mathcal{G}_{A^D}$ is strongly connected, for every pair $(i, j)$ there are $i=i_1, i_2,...,,i_p=j$ such that $\overrightarrow{i_ri}_{r+1}\in \mathcal{G}_{A^D}$, and we have
$$E_{ij}=\displaystyle \prod_{r=1}^{p-1}E_{i_r i_{r+1}}\in \mathfrak{A}(A^D). $$ 
By Burnside's theorem $A^D$ (and, hence, $A$) is irreducible.
\end{proof}

\begin{corollary}\label{components1}
Assume that $A$ is an admissible Hermitian tuple such that $Q_A$ is minimal, and that $D$ is a unitary matrix which diagonalizes $A_1$. If the characteristic graph $\mathcal{G}_{A^D}$ is the union of $r$ strongly connected components, then  $\C^k$ is a direct sum of unitary reducing subspaces for $A$, and the restriction of $A$ to each of these subspaces  is irreducible. 	
\end{corollary}

\begin{proof}
Observe that since  $A^D$ is Hermitian, each pair $(i,j)$ such that $\overrightarrow{ij}\in \mathcal{G}_{A^D}$ is contained in a strongly connected component, and that $\{1,...,k\}$ is a disjoint union of strongly connected components:
$$\{1,...,k\}=\cup_{p=1}^r S_p,  $$
Without loss of generality we may assume that $S_1$ has vertices $1,...,i_1$, \ $S_2$ has $i_1+1,..., i_2$, ... , and $S_r$ has $i_{r-1}+1,...,i_r=k$. 	 

Since $Q_A$ is minimal, it has no repeated factors (Theorem \ref{chacteristic minimal}). Let $\{e_1,...,e_k\}$ be an orthonormal eigenbasis for $A_1^D$.
Based on the arguments in the proof of Theorem \ref{graph connected}, for distinct values $1\leq p,q \leq r$ and any $i\in S_p, \ j\in S_q$, we have 
\begin{equation}\label{amij}
\langle A_m^De_i, e_j\rangle=(A_m^D)_{ji}=0, \ \ \ m=1, ..., n,
\end{equation}
since, if $(A_m^D)_{ji}^m\neq 0$ for some $m$, then $S_p\cup S_q$ would be strongly connected. This implies that each of subspaces
$$ L_j=\SP \{e_{i_{j-1}+1},..., e_{i_j} \}$$
is reducing for $A^D$, and, therefore, $D^*L_j$ is reducing for $A$.. By Theorem \ref{graph connected} the restriction of $A^D$ to each $L_j$ is irreducible, and so is the restriction of $A$ to $D^*L_j$. We are done.
\end{proof}

\begin{remark}
\normalfont Note again that in the proof above we explicitly constructed the reducing subspaces.	
\end{remark}

\begin{remark}
\normalfont If $A$ is a Hermitian tuple, then every matrix in $A$ is diagonalizable. In the case when $Q_A$ has repeated factors, it is not necessary that the strong connectedness of $\mathcal{G}_{A^D}$	 for some diagonalizing matrices D implies that the same is true for all diagonalizations. Nevertheless, we will see in the next section that passing from $A$ to some diagonalization $A^D$ is useful.
\end{remark}

\begin{example}\label{example8.16}
\normalfont Let $A=(A_1,A_2)$ be given by
$$A_1=\mbox{diag}(1,2,3,4), \ A_2=\left( \begin{array}{cccc} 0&0& 1& 0 \\ 0&0&0 &1\\1&0&0&0 \\0&1&0&0\end{array}\right). $$	
In this case $A_1$ is diagonal, so we can take $D=I_k$ and $G_A^D=\mathcal{G}_A$. This graph is presented in Fig. 2 (D), which shows that there are two strongly connected components: one with the vertices $\{1,3\}$, and the other - with $\{2,4\}$. Respectively, span$\{e_1,e_3\}$ and span$\{e_2,e_4\}$ are the corresponding reducing subspaces.
\end{example}

\section{Kippenhahn's Conjecture}\label{Kippenhahn}


In his 1951 paper \cite{Ki}  Kippenhahn came up with the following conjecture that concerns the characteristic polynomials of Hermitian pairs $(A_1, A_2)$. 

\begin{conjecture}
Let $A_1$ and $A_2$ be Hermitian matrices of equal size. If $\det (z_0I+z_1A_1+z_2A_2)$ has a repeated factor, then $(A_1, A_2)$ is reducible.
\end{conjecture}
\noindent In light of Theorem \ref{chacteristic minimal}, this conjecture asserts that if the characteristic polynomial of a Hermitian pair  $(A_1,A_2)$ is not minimal, then the pair is reducible. In \cite{Ki}, Kippenhahn verified the conjecture whenever the degrees $l_j$ in \eqref{char} do not exceed 2.
In subsequent works \cite{Sh1,Sh2}, Shapiro obtained a number of results which supported the conjecture. In particular, she showed that it holds if the size of the matrices $k\leq 5$. In 1983, Laffey \cite{La} showed that, in general, Kippenhahn's conjecture is not true by constructing a counterexample for $k=8$. In the same year, the following counterexamples with $3r\times 3r$ matrices were discovered for all $r\geq 2$ in Waterhouse \cite{Wa}. 

\begin{example}\label{Waterhouse}\normalfont
Let $S$ and $T$ be $r\times r$ ($r\geq 2$) positive definite matrices that have no nontrivial common invariant subspaces and define
\[A_1=\begin{pmatrix} 0 & 0 & 0 \\ 0 & 0 & S \\ 0 & S & 0 \end{pmatrix},
\hspace{1cm} A_2=\begin{pmatrix} 0 & 0 & T \\ 0 & 0 & 0 \\ T & 0 & 0 \end{pmatrix}.\]
Then the characteristic polynomial for the pair $A=(A_1, A_2)$ is
\begin{align*}
    Q_A(z_0,z)&=\det \begin{pmatrix} z_0 & 0 & z_2T \\ 0 & z_0 & z_1S \\ z_2T & z_1S & z_0 \end{pmatrix}
    =\det \begin{pmatrix} z_0 & 0 & z_2T \\ 0 & z_0 & z_1S \\ 0 & z_1S & z_0-\frac{z_2^2}{z_0}T^2 \end{pmatrix}\\
    &=z_0^r\det \begin{pmatrix} z_0 & z_1S \\ z_1S & z_0-\frac{z_2^2}{z_0}T^2 \end{pmatrix}=z_0^r\det(z_0^2-z_1^2S^2-z_2^2T^2),
\end{align*}
which has a repeated factor $z_0$. However, the pair $A$ is irreducible and hence $\mathfrak{A}(A)$ is cyclic. Furthermore, according to Theorem \ref{chacteristic minimal}, $Q_A(z_0,z)$ is not minimal and hence $A$ is not linearly cyclic by Theorem \ref{lcyclic}. Here, we verify this fact by direct computation.

Indeed, for any choice of nonzero vector $z\in \C^2$ and vectors $a, b, c\in \C^r$,
we can choose a nonzero vector $e\in \C^r$ that is orthogonal to $z_1T^{-1}a-z_2S^{-1}b$. Then, direct computation verifies that 
\[\begin{pmatrix}
\bar{z}_1T^{-1}e \\ -\bar{z}_2S^{-1}e \\ 0 \end{pmatrix} \ \text{is orthogonal to}\ A^m_*(z)
\begin{pmatrix} a \\ b \\ c \end{pmatrix}\] for all $m\geq 0$,
showing that $A_*(z)=z_1A_1+z_2A_2$ is not cyclic.
\end{example}

In 1998, more counterexamples with $6\times 6$ matrices were constructed in Li, Spitkovsky, and Shukla \cite{LSS}. Additional counterexamples can be found in \cite{Law}. In 2017, Klep and Vol\u{c}i\u{c} \cite{KVo1} proved the validity of a quantum version of Kippenhahn's conjecture. Recent paper \cite{S2} continued the investigation of the original Kippenhahn's conjecture for Hermitian matrix tuples of arbitrary length whose characteristic polynomial is in the form $(R(z_0,z))^m$ with $m>1$. 
Theorem \ref{kipp main} presents the result obtained there. 

In this section, we will establish a necessary and sufficient condition for the validity of Kippenhahn's conjecture.

\subsection{Spectrally Stable Extensions} 

Our main tool for investigating reducibility of a tuple is constructing some special algebraic extensions, that we call \textit{spectrally stable extensions}. First we will analyze how the local geometry of the joint spectrum changes when additional matrices are added to the tuple.   

 Suppose $A$ is an admissible tuple that is included in an extended tuple
$\widetilde{A}=(A_1,...,A_n,A_{n+1},...,A_N)$ of $k\times k$ matrices, and that 
\begin{equation}\label{extended}
 Q_{A}=\displaystyle \prod_{j=1}^s R_j^{l_j}, \ \ \sigma_p^d(A)=\sum_{j=1}^s l_j \Gamma_j, \ \  Q_{\widetilde{A}}=\displaystyle \prod_{j=1}^q \widetilde{R}_j^{t_j}, \ \  \sigma_p^d(\widetilde{A})=\sum_{j=1}^q t_j \widetilde{\Gamma}_j,
\end{equation}
are the irreducible decompositions of the characteristic polynomials and the proper projective joint spectra of $A$ and $\widetilde{A}$. 

Since $A$ is admissible, for every eigenvalue $\lambda_j, \ j=1,...r$ of $A_1$, the point 
\[\tau_j=(\underbrace{1/\lambda_j,0,...,0}_{n})\] 
belongs to a single component $\Gamma_i\subset \sigma_p(A)$. The point 
\[\widetilde{\tau}_j=\underbrace{1/\lambda_j,0,...,0)}_{N} \] 
belongs to $\sigma_p(\widetilde{A})$, but condition b) of Definition \ref{admissible} is not necessarily satisfied. The possible scenarios for the geometry of $\sigma_p^d(\widetilde{A})$ near $\widetilde{\tau}_j$ are:

\begin{itemize}
\item[a)] $\widetilde{\tau}_j$ belongs to a single component $\widetilde{\Gamma}_m$ and the derivative of $\widetilde{R}_m$ with respect to $z_1$ at $\widetilde{\tau}_j$ does not vanish. In this case the multiplicity of $\widetilde{\Gamma}_m$ is equal to $t_m=l_i$, and for every admissible transformation $C$ of the tuple $\widetilde{A}$ which is close to the identity there is a single eigenvalue of $\widetilde{A}_1=(C\widetilde{A})_i$ that is close to $\lambda_j$, and this eigenvalue has multiplicity $t_m=l_i$.
\item[b)] $\widetilde{\tau}_j$ belongs to a single component $\widetilde{\Gamma}_m$,  but  
$\widetilde{R}_m(z_1,0,...,0)$ has zero of order $\alpha >1$  at $1/\lambda_j$ (as a function of $z_1$).
In this case  $l_i=\alpha t_m$,  
 and for an admissible transformation $C$ of $\widetilde{A}$ there are $\alpha$ eigenvalues of $\widetilde{A}_1=(C\widetilde{A})_1$  close to $\lambda_j$, and each of them has multiplicity $t_m$.
\item[c)] $\widetilde{\tau}_j $ belongs to multiple components $\widetilde{\Gamma}_{i_1}, ...,\widetilde{\Gamma}_{i_p}$, and $\widetilde{R}_{i_m} $ has zero of order $\alpha_m$ at $1/\lambda_j, \ m=1,...,p$. In this case 
$$l_i=\displaystyle \sum_{m=1}^p \alpha_mt_{i_m},$$ 
and for an admissible transformation $C$ of $\widetilde{A}$ that is close to the identity, the matrix $\widetilde{A}_1$ has $\alpha_1+\cdots +\alpha_p$ eigenvalues of multiplicities $t_{i_1},...,t_{i_p}$ respectively.
\end{itemize}

As mentioned above, the multiplicities of spectral components are exactly the multiplicities of eigenvalues of matrices in an admissible tuple. Therefore, we conclude that for every admissible transformation $C$ of $\widetilde{A}$ the number of distinct eigenvalues of $\widetilde{A}_1=(C\widetilde{A})_1$ may only increase compared with the same number of $A_1$, and their corresponding multiplicities may only decrease.

We also see that the situation when the numbers of distinct eigenvalues of $A_1$ and $\widetilde{A}_1$ are the same appears only in item a) above. This shows that, if, when passing to a bigger tuple, the number of distinct eigenvalues of $A_1$ and $\widetilde{A}_1$ remains the same, then their respective multiplicities remain the same, and condition b) of Definition \ref{admissible} is satisfied for $\widetilde{A}_1$.


\begin{definition}
Let $\widetilde{A}$ be an algebraic extension of tuple ${A}$. The extension $A\subset \widetilde{A}$ is said to be {\em spectrally stable } if condition a) above is satisfied at every point $\widetilde{\tau}_j$.
\end{definition}

\begin{remark} \label{spectral stability conditions}
\normalfont 
Observe that the following conditions are equivalent:
\begin{itemize}
\item the extension $A\subset \widetilde{A}$ is spectrally stable;
\item every component $\Gamma_i$ in $\sigma_p^d(A)$ is included in a single component $\widetilde{\Gamma}_R$ of $\sigma_p^d(\widetilde{A})$, and the multiplicities of these components are the same;
\item there exists an admissible transformation of $\widetilde{A}$ in the form
$$ C=\left [\begin{array}{cc} I_n & 0 \\ * & * \end{array} \right ].$$
\end{itemize}
	
\end{remark}

The following facts are immediate consequences of the discussions above.

\begin{proposition}\label{spectral avail}
Let $A$ be an admissible Hermitian tuple. If $\widetilde{A}$ is Hermitian andthe the extension $A\subset \widetilde{A}$ is spectrally stable, then Theorem \ref{local spectral} is applicable at $\widetilde{\tau}_j$.
\end{proposition}

Another straightforward proposition is:

\begin{proposition}\label{intermediate}
If $A\subset \widetilde{A}\subset \widetilde{\widetilde{A}}$ and the extension $A\subset \widetilde{\widetilde{A}}$ is spectrally stable, then so is the extension $A\subset \widetilde{A}$.
\end{proposition}

\vspace{.2cm}

The following result demonstrates a connection between spectral stability of tuple extensions and what we call {\em spectral stability index}.

\begin{definition}\label{spectral index}
The {\em spectral stability index} of a matrix tuple $A$ is defined as
\[\omega(A)=\nu(A)-\kappa(A).\]
\end{definition}

Recall that $s$ is the number of irreducible components in $\sigma_p^d(A)$ (Display (6.10)).

\begin{theorem}\label{omega}
Let $A=(A_1,...,A_n)$ be an admissible  tuple of Hermitian $k\times k$ matrices, $A_1$ be invertible,  and $\widetilde{A}=(A_1,...A_n,A_{n+1},...,A_N)$ be an algebraic extension of $A$. 

a) $\omega(A) \geq \omega({\widetilde{A}})$.

b) If the extension $A\subset \widetilde{A}$ is not spectrally stable, then
$\omega(A) > \omega(\widetilde{A}). $

c) The extension $A\subset \widetilde{A}$ is spectrally stable if and only if there exists a partition $\{1,...,s\}=\cup_{j=1}^q S_j$ such that the component multiplicities $l_i$ are constant for each $S_j$, and 
\small\begin{eqnarray}&\widetilde{R}_j(z_0,z)=\displaystyle \prod_{i\in S_j} R_i(z_0,...,z_n)+\displaystyle \sum_{m=n+1}^N z_mT_{jm}(z_0,...,z_N), \nonumber \\
& \mbox{and} \ t_j=l_i, \ j=1,...,q, \label{partition for stability}
\end{eqnarray}\normalsize
where $T_{jm}$ are polynomials with $deg(T_{jm})=(\sum_{i\in S_j}deg(R_i))-1$.
	\end{theorem}
\begin{proof}
Assume that characteristic polynomials and projective joint spectra of $A$ and $\widetilde{A}$ are given by \eqref{extended}, and let $\lambda_1,...,\lambda_r$ be distinct eigenvalues of $A_1$. Since $A$ is admissible, each point $\tau_j=(1/\lambda_j, 0, ...,0)$ belongs to only 1 component $\Gamma_j\subset \sigma_p(A)$, and the multiplicity of this component is equal to the multiplicity	of $\lambda_j$.
For each component $\Gamma_j$ choose 1 of points $\tau_1,...,\tau_r$ that belongs to $\Gamma_j$. Without loss of generality we may assume that $\tau_j\in \Gamma_j, \ j=1,...,s$. If $\widetilde{\tau}_j\in \widetilde{\Gamma}_i$, then, as we did in scenarios b) and c) in Sect. 9.1, let us denote by 
$\alpha_{ji}$ the order of zero of $\widetilde{R}_m(z_1,0,...,0)$ at $1/\lambda_j$ (as a function of $z_1$). We saw that
$$l_j=\sum_{i: \ \tau_j\in \widetilde{\Gamma}_i} \alpha_{ji}t_i, $$
Further,
$$\omega(A)=\sum_{j=1}^s(l_j-1), \ \ \omega(\widetilde{A})=\sum_{j=1}^q (t_j-1). $$
Each component $\widetilde{\Gamma}_i$ contains at least one of the components $\Gamma_i$. Hence each $\widetilde{\Gamma}_i$ contains at least one point of $\tau_1,...,\tau_s$, so that 
\begin{equation}\label{omega inequality}
\omega(A)= \sum_{j=1}^s(l_j-1)=\sum_{j=1}^s\Big(\sum_{i: \ \tau_j\in \widetilde{\Gamma}_i}\alpha_{ji} t_i -1\Big)\geq \sum_{j=1}^s \sum_{i: \ \tau_j\in \widetilde{\Gamma}_i}( t_i -1) \geq \omega(\widetilde{A}),\
	\end{equation}
which proves a) of this theorem.

Observe that for \eqref{omega inequality} to turn into equality three things must happen:\\
1. We must have $\alpha_{ji}=1$ for all $j,i$ such that $\widetilde{\tau}_j\in \widetilde{\Gamma}_i$. \\
2. Each $\widetilde{\tau}_j$ must belong to a single $\widetilde{\Gamma}_i$, which, of course, implies that every $\Gamma_j$ is included in a single $\widetilde{\Gamma}_i$. \\
3. $t_i=l_j$. 

As mentioned in Remark \ref{spectral stability conditions} these imply that if $\omega(A)=\omega(\widetilde{A})$, the extension $A\subset \widetilde{A}$ is spectrally stable, which proves part b) of this theorem.

Finally, if \eqref{partition for stability} holds, then each component $\Gamma_i$ belongs to a single component $\widetilde{\Gamma}_j$, namely to the one with $i\in S_j$. Since their multiplicities are the same, the extension is spectrally stable. Conversely, if the extension $A\subset \widetilde{A}$ is spectrally stable, then set $S_j=\{ i: \ \Gamma_i \subset \widetilde{\Gamma}_j \}, \ j=1,...,q$. Obviously, $S_j\cap S_l=\emptyset$ if $j\neq l$, and \eqref{partition for stability} holds.
\end{proof}

\vspace{.1cm}

Now, let $A$ be an admissible tuple of Hermitian matrices, and again let $r$ be the number of distinct eigenvalues of $A_1$. Following \eqref{W IJ} for 
$q\leq r$, $2\leq i_1,...,,i_q\leq  n$, and a set of pairwise distinct indices $j_1,...,j_{q-1}\in \{1,...,r\}$  the matrices $W_{IJ}$ were introduced in \eqref{W IJ}:
$$W_{IJ} = A_{i_1}\mathcal{P}_{j_1}A_{i_2}\mathcal{P}_{j_2}\cdots A_{i_q-1}\mathcal{P}_{j_q-1}A_{i_q}.$$
Here $I=(i_1,...,i_q), \ J=(j_1,...j_{q-1})$.  

We denote by $\mathcal{A}_1(A)$ the tuple that includes the original tuple $A$ together with all the indicated words $W_{IJ}$ with $1\leq q\leq r$. Observe that since ${\mathcal P}_j\in \mathfrak{A}(A)$ (see \eqref{projections}), 
the extension $A\to \mathcal{A}_1(A)$ is algebraic. 
Suppose that the characteristic polynomial and the proper projective joint spectrum of $\mathcal{A}_1(A)$ are given by
\begin{equation}\label{characteristic and spectrum 3}
Q_{\mathcal{A}_1(A)}= \widetilde{R}_1 ^{t_1}\cdots \widetilde{R}_m^{t_m}, \ \ \ \sigma_p^d(\mathcal{A}_1(A))=\displaystyle \sum_{j=1}^m t_j\widetilde{\Gamma}_j,
\end{equation}
where once again $\widetilde{R}_j$ are irreducible polynomials, and $\widetilde{\Gamma}_j= \{\widetilde{R}_j=0\}$. We obviously have
$$t_1 \deg(\widetilde{R}_1)+ \cdots +t_m \deg (\widetilde{R}_m)=k. $$

We want to define $\mathcal{A}_j(A)$ for $j>1$. An easy way to do this would be to write $\mathcal{A}_j(A)=\mathcal{A}_1(\mathcal{A}_{j-1}(A))$, but, since  matrices $W_{IJ}$ are not necessarily Hermitian, $\mathcal{A}_1(A)$ is not necessarily a Hermitian tuple. Fortunately, it is easy to circumvent  this obstacle.
\begin{definition} 
Given a tuple $\mathcal{B}$ of $k\times k$ matrices that is invariant under the passing to adjoints, that is, if $B\in \mathcal{B}$, then $B^*\in \mathcal{B}$ we define $\mathcal{B}^s$ as the tuple which is the image of $\mathcal{B}$ under the linear transformation:
\begin{eqnarray*} 
&\mbox{if} \ B=B^*, \ \mbox{then} \ B\mapsto B, \\
&\mbox{if} \ B\neq B^*, \ \mbox{then} \ (B,B^*)\mapsto \big(\frac{1}{2}(B+B^*), \frac{i}{2}(B-B^*)\big).
\end{eqnarray*}
\end{definition}\label{Bs}
It is easily seen that $\mathcal{B}^s$ is a Hermitian tuple, and
\begin{itemize}
\item The multiplicities of irreducible factor of $Q_{\mathcal{B}}$ and $Q_{\mathcal{B}^s}$ are the same;
\item $\mathcal{G}_{\mathcal{B}}$  is strongly connected if and only if $\mathcal{G}_{\mathcal{B}^s}$ is.	
\end{itemize}

Since $A$ is a Hermitian tuple, $\mathcal{A}_1(A)$ is invariant under passing to adjoints. Indeed, for $I=(i_1,...,i_p)$ write $\widehat{I}=(i_p,i_{p-1},...,i_1)$.  We have $(W_{IJ})^*=W_{\widehat{I} \widehat{J}}$. Thus, $(\mathcal{A}_1(A))^s$  is a Hermitian tuple that contains $A_1,...,A_n$, and whose characteristic graph is strongly connected if and only if $\mathcal{G}_{\mathcal{A}_1(A)}$ is.
  
We now define $\mathcal{A}_j(A)$ inductively. Write $\mathcal{A}_0(A)=A$. Suppose that we have already defined $\mathcal{A}_r(A)$ for $r\leq j-1$. We will be using Theorem \ref{adm} that has an assumption of reducibility of the tuple matrices, so we will slightly perturb  $\mathcal{A}_{j-1}(A)^s$.  Recall that we mentioned earlier that we can always  assume $A_1$ to be invertible, so we add to each matrix matrix in $\mathcal{A}_{j-1}(A)^s$ a small real-valued multiple of $A_1$.  Of course, the set of such perturbations that resulted tuple $\Check{\mathcal{A}}_{j-1}(A) $ consists of invertible matrices, is open and dense in the set of all such perturbations. Clearly, $\Check{\mathcal{A}}_{j-1}(A)\subset \mathfrak{A}(A)$ and $ \Check{\mathcal{A}}_{j-1}(A) $ is a Hermitian tuple.
 Next, we choose a close to the identity matrix $C_{j-1}$ that is admissible for $(\Check{\mathcal{A}}_{j-1}(A))^s$ ($C_0=I_n$). 
 and define
\begin{equation}\label{A j}
\mathcal{A}_j(A)=\mathcal{A}_1(C_{j-1}(\Check{\mathcal{A}}_{j-1}(A))^s). \end{equation}

\begin{remark}\label{induct}
\normalfont Matrices in the algebra $\mathfrak{A}(A)$ have  the same lattice of common invariant subspaces as $A$ itself. Since neither an admissible transformation, nor the passing from $A$ to $A^s$ changes the reducibility (or irreducibility) of a tuple, we see that if $\mathcal{A}_j(A)$ is reducible, then $\mathcal{A}_{j-1}$ is. Therefore, if for some $j$ \ the tuple $\mathcal{A}_j(A)$ is  reducible, or irreducible, then  $A$ is either reducible, or irreducible respectively.	
\end{remark}

\vspace{.2cm}

\begin{theorem}\label{repeated multiplicities not 1}
Assume $A$ is an admissible Hermitian tuple. If the extension $A\subset \mathcal{A}_1(A)$ is spectrally stable and not all $t_j$ are equal to 1, then the tuple $A$ is reducible. 
\end{theorem}

\begin{proof}
As before, let us denote by $\lambda_1,...,\lambda_r$ the distinct eigenvalues of $A_1$. We assume that the irreducible spectral decomposition of $Q_A$ is
\begin{equation}\label{spectrum2}
Q_A(z_0,z)=\displaystyle \prod_{j=1}^s (R_j(z_0,z))^{l_j}, \ \sigma_p^d(A)=\sum_{j=1}^sl_j\Gamma_j.
\end{equation}
If $\tau_j=(1/\lambda_j,0,...,0)\in \Gamma_i$, then $rank(\mathcal{P}_j)=l_i$.

There are two different cases.

\textbf{1.} $l_i\neq l_j$ for some $i\neq j$. 

Consider $\widehat{{\mathcal{A}}}_1(A)$, the subtuple of $\mathcal{A}_1(A)$ which includes $A_1,...,A_n$ and only those of $W_{IJ}$ where $q=2$, and $i_1=i_2$,
$$\widehat{\mathcal{A}}_1(A)=\{ A_1,...,A_n, A_2\mathcal{P}_1A_2,...,A_n\mathcal{P}_rA_n\}. $$

 Since $A\subset \widehat{\mathcal{A}}_1(A)\subset \mathcal{A}_1(A)$ Proposition \ref{intermediate} implies that the extension $A\subset \widehat{\mathcal{A}}_1(A)$ is spectrally stable. Therefore, if $\Gamma_j\subset \widehat{\Gamma}_i$, we have $l_j=q_i$.. Let the irreducible  spectral decomposition of $Q_{\widehat{\mathcal{A}}_1(A)}$ be
 \begin{equation}\label{hat A}
 Q_{{\widehat{\mathcal{A}}_1(A)}	}(z_0,z)=\displaystyle \prod_{j=1}^p (\widehat{R}_j(z_0,z))^{q_j}, \ \sigma_p^d(\widehat{\mathcal{A}}_1(A))=\displaystyle \sum_{j=1}^pq_j\widehat{\Gamma}_j.
 \end{equation}

 Since $\widehat{\mathcal{A}}_1(A)$  is a Hermitian tuple, 
Proposition \ref{spectral avail} shows that Theorem \ref{local spectral} is applicable to it. Therefore, for every $2\leq i\leq n$ and every $ j,m=1,...,r, \ j\neq m, $, we have
\begin{equation} \label{block times adjoint}
\mathcal{P}_j(A_i\mathcal{P}_mA_i)\mathcal{P}_j=(\mathcal{P}_jA_i\mathcal{P}_m)(\mathcal{P}_jA_i\mathcal{P}_m)^*=c_{ijm}\mathcal{P}_j,
\end{equation}
where $c_{ijm}\in \R$ is a constant. 

Without loss of generality we may assume that 
$$l_1=l_2=...=l_{u_1}<l_{u_1+1}=....=l_{u_2}<...<l_{u_N}=...=l_s$$ 
 and that
$$\tau_1,...\tau_{i_1}\in \Gamma_1, \ \tau_{i_1+1},...,\tau_{i_2}\in \Gamma_2,\ ..., \ \tau_{i_m+1},...\tau_r\in \Gamma_s.$$

Then 
\begin{eqnarray*}
&rank(\mathcal{P}_1)=\cdots =rank(\mathcal{P}_{i_1})=\cdots =rank(\mathcal{P}_{i_{u_1}})=l_1=\cdots=l_{u_1} \\
&rank(\mathcal{P}_{i_{u_1}+1})=\cdots =rank(\mathcal{P}_{i_{u_2}})=l_{u_1+1}=\cdots =l_{u_2} \\
& \cdot \ \hspace{1.5cm} \cdot \hspace{1.5cm} \ \cdot \ \hspace{1.5cm} \cdot \hspace{1.5cm} \cdot \hspace{1.5cm} \cdot \hspace{1.5cm} \cdot \\
&rank(\mathcal{P}_{i_{u_{N-1}+1}})=\cdots = rank(\mathcal{P}_{r})=l_{u_{N-1}+1}=\cdots=l_s.
\end{eqnarray*}
We will be considering \eqref{block times adjoint} with $i_{u_{q-1}+1}\leq m\leq i_{u_q}$ and $j\geq i_{u_q}+1$. 

If $i_{u_p+1}\leq j\leq i_{u_{p+1}}, \ p\geq q$,   the only non-trivial entries of the $k\times k$ matrix $\mathcal{P}_jA_i\mathcal{P}_m$ are those belonging to the rows corresponding to $\mathcal{P}_j$ and columns corresponding to $\mathcal{P}_m$. We denote this non-zero 

$l_{u_{q}}\times l_{u_p+1}$ block by $u_{ijm}$. Then $u_{ijm}$ is a matrix with the number of rows exceeding the number of columns. Equation \eqref{block times adjoint} yields 
$$ u_{ijm}u_{ijm}^*=c_{ijm}I_{l_{u_q}},$$
where $I_{l_{u_q}}$ is the $l_{u_q}\times l_{u_q}$ identity matrix, and this implies $c_{ijm}=0$. Indeed, if $c_{ijm}\neq 0$, then $l_{u_p+1}$-dimensional vectors represented by the rows of $u_{ijm}$ must form an orthogonal non-trivial system, which is impossible, as $l_{u_q}>l_{u_p+1}$. 

Thus, we see that
\[ L_p=range \big(\mathcal{P}_{i_{u_{p-1}+1}}+ \mathcal{P}_{i_{u_{p-1}+2}}+ \cdots +   \mathcal{P}_{i_{u_p}}\big), \ p=1,...N, \ i_{u_0}=0\]
is invariant under the action of $A_i, \ i=1,...,n$,  so that 
\[\C^k=L_1\oplus L_2 \oplus \cdots \oplus L_N \]
is a decomposition of $\C^k$ into a direct sum of subspaces reducing for $A$.

\vspace{.2cm}

\textbf{2.} $l_i=l_j=t\geq 2$ for all $i,j$.

\vspace{.2cm}

In this case the result follows directly from Theorem \ref{kipp main}. Indeed, 
$$\sigma_p^d(A_1, W_{IJ})=t \displaystyle \sum_{j=1}^s \widetilde{\widetilde{\Gamma}}_j,$$
where $\widetilde{\widetilde{\Gamma}}_j$ is the intersection of $\widetilde{\Gamma}_j$ with the corresponding 2-dimensional plane.
Since $A_1$ is invertible, each $\widetilde{R}_j(z)$ contains the monomial $z_1^{\deg(\widetilde{R})_j}$ with a  non-trivial coefficient, which implies that the condition of Theorem \ref{kipp main} is satisfied.
\end{proof}


\subsection{An Algorithm for Reducibility Check}\label{test1}


In this subsection we describe a procedure of testing reducibility of a Hermitian tuple..

Recall that given a Hermitian tuple $B$ with $B_1$ being invertible, we denoted above by $\Check{B}$ a perturbed tuple obtained from $B$ by adding to each element of $B$ a small real-valued multiple of $B_1$ to make matrices in $\Check{B}$ invertible.

Suppose that $A$ is an admissible Hermitian tuple, and $Q_A$ is given by \eqref{spectrum2}.  For $j\geq 0$ the tuple $\mathcal{A}_j(A)$ is given by \eqref{A j}. We repeat the following 5 step cycle until the algorithm stops.

\textbf{The $j$th cycle}

\textbf{1}. If $Q_{(\mathcal{A}_{j-1}(A))^s}$ has no repeated factors, and, hence, neither has \newline $Q_{C_{j-1}(\Check{\mathcal{A}}_{j-1}(A))^s}$,  choose a matrix $D_{j-1}$ that diagonilizes $(C_{j-1}(\Check{\mathcal{A}}_{j-1}(A))^s)_1$. 
\begin{itemize}
\item[a)] If $\mathcal{G}_{(C_{j-1}(\Check{\mathcal{A}}_{j-1}(A))^s)^{D_{j-1}}}$ is strongly connected, Theorem \ref{graph connected} and Remark \ref{induct} imply that $A$ is irreducible and the algorithm stops.
\item[b)] If $\mathcal{G}_{(C_{j-1}(\Check{\mathcal{A}}_{j-1}(A))^s)^{D_{j-1}}}$ is not strongly connected, 	Theorem \ref{graph connected} and Remark \ref{induct}  imply that $A$ is reducible, and the algorithm stops.
\end{itemize}

\textbf{2}. If $Q_{C_{j-1}(\Check{\mathcal{A}}_{j-1}(A))^s}$ has repeated factors, and  $ \mathcal{G}_{(C_{j-1}(\Check{\mathcal{A}}_{j-1}(A))^s}$ is not strongly connected, $A$ is reducible by Theorem \ref{not connected} and Remark \ref{induct}, and the algorithm stops.

\textbf{3}. If $Q_{C_{j-1}(\Check{\mathcal{A}}_{j-1}(A))^s}$ has repeated factors, and  $ \mathcal{G}_{(C_{j-1}(\Check{\mathcal{A}}_{j-1}(A))^s}$ is  strongly connected, we pass to $(\mathcal{A}_j(A))$. If $Q_{\mathcal{A}_j(A))}$ has no repeated factors, then neither has $Q_{C_j(\Check{\mathcal{A}}_j(A) )^s }$. Choose a matrix $D_j$ that diagonalizes $(C_j(\Check{\mathcal{A}}_j(A))^s)_1$.

\begin{itemize}
\item[a)] If $\mathcal{G}_{(C_j(\Check{\mathcal{A}}_j(A))^s)^{D_j}}$  is strongly connected, Theorem \ref{graph connected} implies that 
 \newline $C_j(\Check{\mathcal{A}}_j(A))^s $ is irreducible. Therefore, $(\mathcal{A}_j(A))^s $, is irreducible. Remark \ref{induct} shows that $A$ is irreducible as well, and the algorithm stops.
\item[b)]. If $\mathcal{G}_{(C_j(\Check{\mathcal{A}}_j(A))^s)^{D_j}}$  is not strongly connected, Theorem \ref{graph connected} implies that $C_j(\Check{\mathcal{A}}_1(A))^s $ is reducible, and, once again, by Remark \ref{induct} this yields the reducibility of $A$.
\end{itemize}

\textbf{4}. Let $Q_{\mathcal{A}_j(A)}$ have  repeated factors.  Theorem \ref{omega} facilitates the determination whether the extension $C_{j-1}(\Check{\mathcal{A}}_{j-1}(A))^s \subset \mathcal{A}_j(A)$ is spectrally stable. If it is,  Theorem \ref{repeated multiplicities not 1} and Remark \ref{induct} show that $A$ is reducible, and the algorithm stops. 
 
\textbf{5}.  If $Q_{\mathcal{A}_j(A)}$ have repeated factors. and the extension $C_{j-1}(\Check{\mathcal{A}}_{j-1}(A))^s \subset \mathcal{A}_j(A)$ is not spectrally stable, we go the next $(j+1)$st cycle.

\vspace{.2cm}

\vspace{.2cm}

If $\omega(A)=0$, then $Q_A$ has no repeated factors, and the above algorithm reaches the conclusion during the first cycle.

\begin{theorem}\label{test}
Let $A$ be an admissible Hermitian tuple, and $A_1$ be invertible. If $\omega(A)\geq 1$, then the reducibility check reaches a conclusion in no more than $\omega(A)$ cycles.	
\end{theorem}
\begin{proof}
It is easily seen that $\omega(C_{j-1}(\Check{\mathcal{A}}_{j-1}(A))^s)=\omega(\Check{\mathcal{A}}_{j-1}(A))=\omega(\mathcal{A}_{j-1}(A))$. Theorem \ref{omega}	shows that, if the $j$th cycle of the reducibility check was inconclusive, and we pass to the $(j+1)$st cycle, then $\omega(\mathcal{A}_j(A))<\omega(\mathcal{A}_{j-1}(A))$. Thus, if the algorithm did not stop after $\omega(A)-1$ cycles, then $\omega(\mathcal{A}_{\omega(A)-1}(A))=1$. Then, either we will have $\omega(\mathcal{A}_{\omega(A)})=1$, and by Theorem \ref{omega} the extension $C_{\omega(A)-1}(\Check{\mathcal{A}}_{\omega(A)-1}(A))^s \subset \mathcal{A}_{\omega(A)}(A)$ is spectrally stable, so that the $\omega(A)$th cycle stop on step 4, or $\omega(\mathcal{A}_{\omega(A)})=0$, which means that $Q_{\mathcal{A}_{\omega(A)}}$ has no repeated roots, and the $\omega(A)$th cycle stops on step 3. In both cases the $\omega(A)$th cycle is conclusive.
\end{proof}


To help illustrate Theorem \ref{test}, let us take another look at Example \ref{Waterhouse}. Assume $S$ and $T$ satisfy the additional condition that $\det(z_0^2-z_1^2S^2-z_2^2T^2)$ is irreducible (the generic case). Then $\kappa(A)=2, \ \nu(A)=r+1, \ \omega(A)=r-1$, and it follows that the check of reducibility  reaches a conclusion in at most $r-1$ cycles. 

\vspace{.2cm}

\subsection{Kippenhahn's Conjecture: Necessary and Sufficient Conditions of Validity} 

Now, we are ready to give a necessary and sufficient condition for a positive answer in Kippenhahn's conjecture for Hermitian tuples of arbitrary length. 

\begin{theorem}\label{kippen}
Assume $A$ is a Hermitian tuple such that $A_1$ is invertible, and $Q_A$ has repeated factors. Let $C$ be an admissible transformation of $\Check{A}$, and $\widetilde{A}=C\Check{A}$. 
Then $A$ is reducible if and only if either for some $0\leq j\leq \omega(A)$, the characteristic graph $\mathcal{G}_{(C_{j}(\Check{\mathcal{A}}_{j}(\widetilde{A}))^s)^{D_{j}}}$ is not strongly connected, or for some $1\leq j \leq \omega(A)$ the characteristic polynomial $Q_{\mathcal{A}_j(\widetilde{A})}$ has repeated factors and the extension $C_{j-1}(\Check{\mathcal{A}}_{j-1}(\widetilde{A}))^s \subset \mathcal{A}_{j}(\widetilde{A})$ is spectrally stable.
\end{theorem}
The proof of this theorem follows directly from the results of subsection \ref{test1}

 
\begin{example} \normalfont Let $A=(A_1,A_2)$ be an admissible pair of $2m\times 2m$ invertible Hermitian matrices, and suppose that $$ Q_A= \big (R(z_0,z_1,z_2)\big)^2,$$ where $R$ is an irreducible polynomial of degree $m$.  The matrix $A_1$ has $m$ distinct eigenvalues, each of multiplicity 2. 
By Theorem \ref{test}  to determine whether the tuple is reducible we need a single  cycle.

$Q_A$ has repeated factors, so we pass to step 2.
If $\mathcal{G}_A$ is not strongly connected, $A$ is reducible.

If $\mathcal{G}_A$ is strongly connected we go through step 3 and consider $\mathcal{A}_1(A)$.
\begin{itemize}
\item[a)] $\sigma_p^d(\mathcal{A}_1(A))=\big \{ \mathcal{R}(z)=0\big\}	$, where $\mathcal{R}$ is an irreducible polynomial of degree $2m$. In this case the pair $(A_1,A_2)$ is irreducible.
\item[b)] $\sigma_p^d(\mathcal{A}_1(A))= \big\{ \mathcal{R}_1(z)=0\big\}+\big\{ \mathcal{R}_2(z)=0\big\}$, where both $\mathcal{R}_1$ and $\mathcal{R}_2$ are irreducible polynomials of degree $m$. In this case $\omega(\mathcal{A}_1(A))=0$. Choose an admissible matrix $C$ for $\Check{\mathcal{A}}_1(A)$ and a matrix $D$ that diagonalizes $(C\Check{\mathcal{A}}_1(A))_1$ .	Theorem \ref{graph connected} implies that if the characteristic graph $\mathcal{G}_{(C\Check{\mathcal{A}}_1(A))^D}$ is strongly connected, then $A$ is irreducible. Otherwise it is reducible, and the proof of Corollary \ref{components1} shows how to  find reducing subspaces.
\item[c)]  $\sigma_p^d(\mathcal{A}_1(A))=2\big\{ (\mathcal{R}(z))=0\big\} $, where $\mathcal{R}$ is an irreducible polynomial of degree $m$. In this case $\omega(A)=\omega(\mathcal{A}_1(A))=1$, and Theorem \ref{omega} shows that the extension $A\subset \mathcal{A}_1(A)$ is spectrally stable, and according to Theorem \ref{repeated multiplicities not 1} $(A_1,A_2)$ is reducible.
\end{itemize}

To identify a reducing subspace in case c) above we follow the method described in the proof of Theorem 3.1 in \cite{S2}.

Once again, let $\lambda_1,...,\lambda_m$ be distinct eigenvalues of $A_1$ each of multiplicity 2, and $D$ be a unitary matrix that diagonalizes $A_1$, so that 
$$A_1^D=\diag(\lambda_1,\lambda_1,\lambda_2,\lambda_2,...,\lambda_m,\lambda_m)$$. 
In our case the proof of Theorem 3.1in \cite{S2} shows that $A_2^D$ is written in the form
$$ A_2^D=\left [\begin{array}{ccc} c_{11}u_{11} & ...& c_{1m}u_{1m}\\ \cdot & \cdot & \cdot \\ c_{n1}u_{m1} & ... &c_{mm}u_{mm}  \end{array}  \right ],$$
where $c_{ij}\in \C$, and $u_{ij}$ are unitary $2\times 2$ matrices satisfying the following relations:

\begin{itemize}
\item[1.] $u_{ii}=I_2$,
\item[2.] $u_{ji}=u_{ij}^*, \ c_{ji}=\overline{c_{ij}} $,
\item[3.] $u_{i_1 i_2}u_{i_2i_3}\cdots u_{i_{p-1} i_p}=e^{i\theta}u_{i_1 i_p}$. 
\end{itemize}

 Also, the proof of Theorem 3.1 in \cite{S2} shows that in our case, if $U$ is the following unitary  block-diagonal matrix with $2\times 2$ diagonal blocks
$$ U=\left [ \begin{array}{ccccc} u_{m1}&0 &...& 0&0\\ 0 & u_{m2}& ...& 0&0\\ \cdot & \cdot & \cdot & \cdot &\cdot \\ 0&\cdot &\cdot & u_{m (m-1)} &0\\ 0& \cdot &\cdot& 0 & I_2\end{array} \right ],$$
then $L=\text{span}\{U^*e_1, U^*e_3,...,U^*e_{2m-1}\}$ is reducing for both $A_1^D$ and $A_2^D$, and $D^*L$ is reducing for $A_1$ and $A_2$.
\end{example}
\vspace{.2cm}


\section{Concluding Remarks}\label{concluding}

This paper has developed a framework connecting the reducibility of linear representations and matrix tuples to the geometric and algebraic properties of their joint characteristic polynomials.
It extends the classic Cayley-Hamilton theorem and the concept of minimal polynomial to several variables. Crucially, by combining this work with local spectral analysis, this paper provides necessary and sufficient condition for the validity of Kippenhahn's  1951 conjecture generalized to tuples of arbitrary length.

The interplay between algebraic extensions, spectral geometry, and representation theory opens compelling pathways for subsequent study: 

1) The minimal tuple extension and word length problem (Problem \ref{grp-extension}) calls for deeper analyses on the 130-year old Frobenius theorem.

2) For admissible matrix tuples $A$ such that $Q_A$ has no repeated factors, the characteristic graph offers an efficient mechanism for determining their reducibility and constructing common invariant subspaces (Theorem \ref{graph connected}). It is worth studying whether this method can be extended to more general tuples.

3) While this paper bridges classical work on group determinants with modern multivariable operator theory, the characteristic polynomial approach remains largely unexplored in other theoretical contexts where noncommuting matrices play fundamental roles, such as quantum physics, engineering, and artificial intelligence. Accordingly, investigating its potential applications to these domains will be a significant and timely undertaking.

\vspace{.5cm}

\noindent {\bf AI Use Disclosure}: GPT 6.0 Astra and Lean were used to check the proofs of the main results, and they also contributed in part to the proofs of Lemma 7.6 and Theorem 7.9. All other work was done by the authors.

\end{document}